\documentclass[leqno,preprint,12pt]{elsarticle}

 \usepackage{graphicx}

\makeatletter
  
  \@addtoreset{equation}{section}
\makeatother

\usepackage{upref,amsxtra,amscd,amsopn,amsbsy,amstext,amssymb,amsmath,amsthm}
\usepackage{bm}

\newtheorem{thm}{Theorem}[section]
\newtheorem{lem}{Lemma}[section]

\newtheorem{rem}{Remark}[section]

\newtheorem{dfn}{Definition}[section]

\newcommand{\pd}{\partial}
\newcommand{\gm}{\gamma}
\newcommand{\Gm}{\Gamma}
\newcommand{\bn}{\bm{\nu}}

\newcommand{\lm}{\lambda}

\newcommand{\Gl}{G^\lm}
\newcommand{\R}{\mathbb{R}}

\newcommand{\N}{\mathbb{N}}

\newcommand{\su}[2]{\sup_{#1}{#2}}

\newcommand{\av}[1]{\left| #1 \right|}
\newcommand{\Ip}[2]{\left\langle #1 , #2 \right\rangle}

\newcommand{\Ln}[2]{\left\| #1 \right\|_{L^2(#2)}}

\newcommand{\Lp}[2]{\left\| #1 \right\|_{L^{#2}(0,L)}}
\newcommand{\Lns}[2]{\left\| #1 \right\|_{L^{#2}_s}}

\newcommand{\Wn}[3]{\left\| #1 \right\|_{W^{#2}_{#3}(0,L)}}

\newcommand{\vk}{\kappa}
\newcommand{\vp}{\varphi}

\newcommand{\va}{\alpha}
\newcommand{\vb}{\beta}
\newcommand{\vd}{\delta}
\newcommand{\vz}{\zeta}
\newcommand{\vs}{\theta}

\newcommand{\mL}{\mathcal{L}}
\newcommand{\mE}{\mathcal{E}}
\newcommand{\mfq}{\mathfrak{q}}

\newcommand{\ve}{\varepsilon}

\newcommand{\lmt}[2]{\lim_{#1}{#2}}

\journal{Journal of Differential Equations}

\begin{document}

\begin{frontmatter}



\title{Curve shortening-straightening flow for \\ non-closed planar curves with infinite length}


\author{Matteo Novaga}
\address{Department of Mathematics, University of Padova, Via Trieste 63, 35121 Padova, Italy}

\author{Shinya Okabe}
\address{Mathematical Institute, Tohoku University, Sendai 980-8578, Japan}

\begin{abstract}
We consider a motion of non-closed planar curves with infinite length. 
The motion is governed by a steepest descent flow for the geometric functional 
which consists of the sum of the length functional and the total squared curvature. 
We call the flow shortening-straightening flow. 
In this paper, first we prove a long time existence result for the 
shortening-straightening flow for non-closed planar curves with infinite length. 
Then we show that the solution converges to a stationary solution as time 
goes to infinity. Moreover we give a classification of the stationary solution. 
\end{abstract}


\end{frontmatter}



\section{Introduction}
There are various studies about the steepest descent flow for geometric functional defined on 
closed curves, for example, the shortening flow (\cite{ange}, \cite{gage}, \cite{gray}), 
the straightening flow for curve with fixed total length (\cite{linner}, \cite{wen-1}, \cite{wen-2}), 
and the straightening flow for curve with fixed local length (\cite{koiso}, \cite{okabe-1}). 
In this paper, we consider the steepest descent flow called shortening-straightening flow. 

Let $\gm$ be a planar curve, $\vk$ be the curvature, and $s$ denote the arc-length parameter of $\gm$. 
For $\gm$, we consider the following geometric functional 
\begin{align} \label{s-s-energy}
E(\gm) = \lm^2 \mL(\gm) + \mE(\gm), 
\end{align}
where
\begin{align*}
\mL(\gm)= \int_\gm ds, \qquad 
\mE(\gm)= \int_\gm \vk^2 \, ds, 
\end{align*}
and $\lm$ is a given non-zero constant. 
The steepest descent flow for \eqref{s-s-energy} is given by the system 
\begin{align} \label{s-s-flow}
\pd_t \gm = ( - 2 \pd^2_s \vk - \vk^3 + \lm^2 \vk ) \bn, 
\end{align}
where $\bn$ is the unit normal vector of the curve pointing in the direction of the curvature. 
The functional $\mL(\gm)$ denotes the length functional of $\gm$. 
We call the steepest descent flow for length functional the curve shortening flow. 
On the other hand, the functional $\mE$ is well known as the total squared curvature 
or one dimensional Willmore functional. 
The steepest descent flow for the functional is called the curve straightening flow. 
Thus we call \eqref{s-s-flow} the shortening-straightening flow in this paper. 

We mention the known results of shortening-straightening flow. 
In 1996, it has been proved by A. Polden (\cite{polden}) that the equation \eqref{s-s-flow} admits smooth 
solutions globally defined in time, 
when the initial curve is closed and has finite length (i.e., compact without boundary). 
Furthermore, G. Dziuk, E. Kuwert, and R. Sch\"atzle (\cite{dziuk}) extended the result of 
\cite{polden} to closed curves with finite length in $\R^n$. 

We are interested in the following problem: 
``What is the dynamics of {\it non-closed} planar curve with {\it infinite} length 
governed by shortening-straightening flow?" 
In this paper, we prove that there exists a long time solution of shortening-straightening flow 
starting from smooth planar curve with infinite length. 
Moreover we show that the solution converges to a stationary solution as $t \to \infty$. 
Namely, we consider the following initial value problem: 
\begin{align} \label{s-s}
\begin{cases}
& \pd_t \gm  = (- 2 \pd^2_s \vk - \vk^3 + \lm^2 \vk ) \bn, \\
& \gm(x,0)= \gm_0(x).   
\end{cases} \tag{SS}
\end{align}
The initial curve $\gm_0$ is a smooth non-closed planar curve with infinite length. 
Moreover we assume that $\gm_0$  is allowed to have self-intersections but must be 
close to an axis in a $C^1$ sense as $\av{x} \to \infty$. 
More precisely, $\gm_0(x)= (\phi_0(x), \psi_0(x)) : \R \to \R^2$ satisfies the following assumptions: 
\begin{gather} 
\av{{\gm_0}'(x)} \equiv 1, \label{cond-1-i} \\
\pd^m_x \vk_0 \in L^2(\R) \quad \text{\rm for all} \,\,\,\,  m \geq 0, \label{cond-2-i} \\ 
\lmt{x \to \infty}{\phi_0(x)}= \infty, \quad 
\lmt{x \to -\infty}{\phi_0(x)}= -\infty,\quad 
\lmt{\av{x} \to \infty}{\phi'_0(x)}=1, \label{cond-3-i} \\ 
\psi_0(x) = O(x^{-\va}) \,\,\,\text{for \,some}\,\,\,\va > \frac{1}{2} \,\,\, \text{as} \,\,\, \av{x} \to \infty,\quad  
\psi'_0 \in L^2(\R),  \label{cond-4-i}
\end{gather}

We state the main result of this paper in a concise form: 

\vspace{0.1cm}

\begin{thm} \label{main-thm}
Let $\gm_0(x)$ be a planar curve satisfying \eqref{cond-1-i}--\eqref{cond-4-i}. 
Then there exist a family of smooth planar curves 
$\gm(x,t) : \R \times [0, \infty) \to \R^2$ satisfying \eqref{s-s}. 
Moreover, there exist sequences $\{ t_j \}^\infty_{j=1}$ 
and $\{ p_j \}^\infty_{j=1}$ and a smooth curve $\hat{\gm} : \R \to \R^2$ such that 
$\gm(\cdot,t_j) - p_j$ converges to $\hat{\gm}(\cdot)$ as $t_j \to \infty$ up to 
a reparametrization. The curve $\hat{\gm}$ satisfies $\mE(\hat{\gm})<\infty$ and 
the curvature $\hat{\vk}$ is a solution of 
\begin{align*}
2 \pd^2_{s} \hat{\vk} + \hat{\vk}^3 - \lm^2 \hat{\vk}=0. 
\end{align*}
\end{thm}

\vspace{0.1cm}

Generally, in order to prove a long time existence of a steepest descent flow for a functional, 
we have to make use of a priori boundedness which proceeds from the functional. 
Thus the functional must be bounded at least for an initial state. 
However our functional $E$ is unbounded, because we consider planar curves with infinite length. 
This is a difficulty of our problem. 
One of the contribution of Theorem \ref{main-thm} is to prove a long time existence of 
the steepest descent flow for the unbounded functional $E$. 
In order to overcome the difficulty we mentioned above, 
we construct the solution of \eqref{s-s} by making use of Arzel\`a-Ascoli's theorem. 
To define a sequence approximating a solution of \eqref{s-s}, 
we need to solve a certain compact case with fixed boundary. 

Concerning the classification of stationary state, one of the types is a straight line. 
This is corresponding to a trivial stationary state. 
On the other hand, the other one corresponds to a non-trivial stationary state. 
We give not only a classification but also a characterization of them 
(see Theorem \ref{main-thm-2}). 
Although a dynamical aspect of solution of \eqref{s-s} is an open problem, 
to classify and to characterize the stationary state is an important step 
to comprehend the dynamics. 

The paper is organized as follows: In Section \ref{c-c-w-f-b}, we prove that, 
for a non-closed planar curve with finite length, 
there exists a unique long time classical solution of \eqref{s-s-flow} with 
certain boundary conditions. 
Furthermore we show that the solution converges to a stationary solution along a 
sequence of time $\{ t_j \}_j$ with $t_j \to \infty$. 
In Section \ref{n-comp-case}, we prove (i) a long time existence of solution of \eqref{s-s} 
and a certain asymptotic profile of the solution as $\av{x} \to \infty$ 
(Theorem \ref{main-thm-1}), 
(ii) a subconvergence of the solution to a stationary solution, 
(iii) a classification of the stationary solutions (Theorem \ref{main-thm-2}), 
and (iv) a characterization of a dynamical aspect of the solution of \eqref{s-s} 
(Theorem \ref{main-thm-3}).   

\section{Compact case with fixed boundary} \label{c-c-w-f-b}
Let $\Gm_0(x) : [0,L] \to \R^2$ be a smooth planar curve and 
$k_0(x)$ denote the curvature. 
Let $\Gm_0(x)$ satisfy 
\begin{align} \label{i-cond-cc}
\av{{\Gm_0}'(x)} \equiv 1, \quad 
\Gm_0(0) = (0, 0), \quad 
\Gm_0(L)=(R,0), 
\quad k_0(0)= k_0(L)=0, 
\end{align}
where $L>0$ and $R>0$ are given constants. 
We consider the following initial boundary value problem: 
\begin{align} \label{comp-sys}
\begin{cases}
& \pd_t \gm  = (- 2 \pd^2_s \vk - \vk^3 + \lm^2 \vk ) \bn, \\
& \gm(0,t) = (0,0), \quad \gm(L,t)  = (R, 0), \quad 
  \vk(0,t)= \vk(L,t) = 0, \\ 
& \gm(x,0)=\Gm_0(x)  
\end{cases} \tag{CSS}
\end{align}
The purpose of this section is to prove the following theorem: 
\begin{thm} \label{comp-gm-th}
Let $\Gm_0$ be a smooth planar curve satisfying the condition \eqref{i-cond-cc}. 
Then there exists a unique classical solution of \eqref{comp-sys} for any time $t > 0$. 
\end{thm}
\subsection{Short time existence} \label{s-exe}
First we show a short time existence of solution to \eqref{comp-sys}. 
Let 
\begin{align} \label{d-d}
\gm(x,t) = \Gm_0(x) + d(x,t) \bn_0(x),  
\end{align}
where $d(x,t) : [0,L] \times [0, \infty) \to \R$ is an unknown scalar function and 
$\bn_0(x)$ is the unit normal vector of $\Gm_0(x)$, 
i.e., $\bn_0(x)= \mathfrak{R} {\Gm_0}'(x) 
= \bigl( \begin{smallmatrix} 0 & -1 \\ 1 & 0\end{smallmatrix} \bigr) {\Gm_0}'(x)$. 
Under the formulation \eqref{d-d}, the boundary conditions 
$\gm(0,t)=(0,0)$ and $\gm(L,t)=(R,0)$ are reduced to  
\begin{align} \label{d-bdc-1}
d(0,t)= d(L,t)=0. 
\end{align}
With the aid of Frenet-Serret's formula ${\Gm_0}''= k_0 \bn_0$ and 
${\bn_0}' =-k_0 {\Gm_0}'$, we have   
\begin{align*}
\pd_x \gm &= (1 - k_0 d) {\Gm_0}' + \pd_x d \bn_0, \\
\mathfrak{R} \pd_x \gm &= -\pd_x d {\Gm_0}' + (1 - k_0 d)\bn_0, \\
\pd^2_x \gm &= (- {k_0}' d - 2 k_0 \pd_x d) {\Gm_0}' 
                + (\pd^2_x d + k_0 - {k_0}^2 d ) \bn_0, \\ 
\vk &= \dfrac{\pd^2_x \gm \cdot \mathfrak{R} \pd_x \gm}{\av{\pd_x \gm}^3} 
     = \dfrac{\pd_x d ( {k_0}' d + 2 k_0 \pd_x d) + (1 - k_0 d)(\pd^2_x d + k_0 - {k_0}^2 d)}
             {\left\{ (1 - k_0 d)^2 + (\pd_x d)^2 \right\}^{3/2} } . 
\end{align*}
Thus the condition $\vk(0,t)= \vk(L,t)=0$ implies 
\begin{align} \label{d-bdc-2}
\pd^2_x d(0,t) = \pd^2_x d(L,t)=0. 
\end{align}
Let $s=s(x,t)$ denote the arc length parameter of $\gm(x,t)$. Since 
\begin{align*}
s(x,t) = \int^x_0 \av{\pd_x \gm(x,t)}\, dx
       = \int^x_0 \left\{ (1 - k_0(x) d(x,t))^2 + (\pd_x d(x,t))^2 \right\}^{1/2}\, dx, 
\end{align*}
we have 
\begin{align} \label{ps-px}
\dfrac{\pd s}{\pd x} := \av{\gm_d} 
 = \left\{ (1 - k_0(x) d(x,t))^2 + (\pd_x d(x,t))^2 \right\}^{1/2}. 
\end{align}
Combining the relation \eqref{ps-px} with  
\begin{align*}
\dfrac{\pd}{\pd s}
 = \dfrac{\pd/\pd x}{\pd s/ \pd x}, 
\end{align*}
we obtain 
\begin{align*}
\dfrac{\pd}{\pd s} 
 = \dfrac{\pd_x}{\av{\gm_d}}. 
\end{align*}
Then we see that 
\begin{align*}
\pd^2_s \vk 
 = \dfrac{\pd_x}{\av{\gm_d}} \left( 
    \dfrac{\pd_x}{\av{\gm_d}} \left( 
     \dfrac{\pd_x d (\pd_x k_0 d + 2 k_0 \pd_x d) + (1 - k_0 d)(\pd^2_x d + k_0 - {k_0}^2 d)}
             {\av{\gm_d}^3}
      \right) \right). 
\end{align*}
This is reduced to 
\begin{align*}
\partial^2_s \vk 
 = \dfrac{1}{\av{\gm_d}^5} \pd^2_x \va_3 
    - \dfrac{7}{\av{\gm_d}^6} \pd_x \av{\gm_d} \pd_x \va_3 
     + \left\{ -\dfrac{3}{\av{\gm_d}^6} \pd^2_x \av{\gm_d} 
               + \dfrac{15}{\av{\gm_d}^7} \left( \pd_x \av{\gm_d} \right)^2 \right\} \va_3, 
\end{align*}
where 
\begin{align*}
\va_3= \pd_x d (\pd_x k_0 d + 2 k_0 \pd_x d) + (1 - k_0 d)(\pd^2_x d + k_0 - {k_0}^2 d). 
\end{align*}
Setting 
\begin{align*}
\va_1 &= \pd_x k_0 d + k_0 \pd_x d, \\
\va_2 &= \pd_x d \pd^2_x d + \va_1 (k_0 d - 1), \\
\va_4 &= \pd_x d \pd^3_x d + (\pd^2_x d)^2 + {\va_1}^2 
           + \pd_x \va_1 (k_0 d - 1), 
\end{align*}
we have 
\begin{align*}
\pd_x \av{\gm_d} = \dfrac{\va_2}{\av{\gm_d}}, \qquad 
\pd^2_x \av{\gm_d}= -\dfrac{{\va_2}^2}{\av{\gm_d}^3} 
                          + \dfrac{\va_4}{\av{\gm_d}}. 
\end{align*}
Thus $\partial^2_s \vk$ is written as 
\begin{align*}
\pd^2_s \vk 
 = \dfrac{1}{\av{\gm_d}^5} \pd^2_x \va_3 
    - \dfrac{1}{\av{\gm_d}^7}\left( 7 \va_2 \pd_x \va_3 + 3 \va_3 \va_4 \right)
     + \dfrac{18}{\av{\gm_d}^9} {\va_2}^2 \va_3.  
\end{align*}
Since $\vk= \va_3/\av{\gm_d}^3$ and 
\begin{align*}
\pd_t \gm = \pd_t d \bn_0,  
\end{align*}
we have 
\begin{align*}
\partial_t d &= 
 \left\{-\dfrac{2}{\av{\gm_d}^4} \partial^2_x \va_3 + \dfrac{14}{\av{\gm_d}^6} \va_2 \partial_x \va_3 
    + \dfrac{6}{\av{\gm_d}^6} \va_3 \va_4 - \dfrac{36}{\av{\gm_d}^8}{\va_2}^2 \va_3 
     -\dfrac{{\va_3}^3}{\av{\gm_d}^8} + \dfrac{\lm^2 \va_3}{\av{\gm_d}^2} \right\}
  \dfrac{1}{1 - k_0 d} \\
& = - \dfrac{2}{\av{\gm_d}^4} \pd^4_x d + \Phi(d). 
\end{align*}
Setting $A(d) = (-2/\av{\gm_d}^4) \pd^4_x$, 
the problem \eqref{comp-sys} is written in terms of $d$ as follows: 
\begin{align} \label{d-sys}
\begin{cases}
& \pd_t d  = A(d) d + \Phi(d), \\
& d(0,t) = d(L,t) = d''(0,t) = d''(L,t) =0, \\ 
& d(x,0)= d_0(x)= 0. 
\end{cases}
\end{align}
We find a smooth solution of \eqref{d-sys} for a short time. 
To do so, we need to show the operator $A_0:=A(d_0)$ is a sectorial operator. 
Since $A_0 = -2 \pd^4_x$, first we consider the boundary value problem 
\begin{align} \label{bd-prob-1}
\begin{cases}
& \pd^4_x \vp + \mu \vp = f, \\
& \vp(0)=\vp(L)=\vp''(0)=\vp''(L)=0,  
\end{cases}
\end{align}
where $\mu$ is a constant. 
The solution of \eqref{bd-prob-1} is written as 
\begin{align} \label{sol-ex}
\vp(x)= \int^L_0 G(x,\xi) f(\xi) \, d \xi, 
\end{align}
where $G(x,\xi)$ is a Green function given by 
\begin{align} \label{G-ex}
G(x,\xi) = 
\begin{cases}
\vspace{0.2cm}
\dfrac{1}{(2 \mu_*)^3} (g_1(\xi) g_2(x) + g_3(\xi) g_4(x)) \quad & \text{for} \quad 0 \leq x \leq \xi, \\
\dfrac{1}{(2 \mu_*)^3} (g_1(x) g_2(\xi) + g_3(x) g_4(\xi)) \quad & \text{for} \quad \xi < x \leq L. \\
\end{cases}
\end{align}
Here the functions $g_1$, $g_2$, $g_3$, $g_4$, and constants $K_0$, $K_1$, $K_2$, $\mu_*$ are given by 
\begin{align*}
g_1(\vz)&= \cos{\mu_* \vz} \sinh{\mu_* \vz} 
                     - \sin{\mu_* \vz} \cosh{\mu_* \vz}, \\ 
g_2(\vz)&= e^{\mu_* \vz} \cos{\mu_* \vz} 
            -\dfrac{K_1}{K_0} \cos{\mu_* \vz} \sinh{\mu_* \vz} 
            +\dfrac{K_2}{K_0} \sin{\mu_* \vz} \cosh{\mu_* \vz}, \\
g_3(\vz)&= \cos{\mu_* \vz} \sinh{\mu_* \vz} 
                     + \sin{\mu_* \vz} \cosh{\mu_* \vz}, \\ 
g_4(\vz)&= -e^{\mu_* \vz} \sin{\mu_* \vz} 
            +\dfrac{K_1}{K_0} \sin{\mu_* \vz} \cosh{\mu_* \vz} 
            +\dfrac{K_2}{K_0} \cos{\mu_* \vz} \sinh{\mu_* \vz}, \\
K_0 &= 2 \cos^2{\mu_* L} \sinh^2{\mu_* L} 
        + 2 \sin^2{\mu_* L} \cosh^2{\mu_* L}, \\
K_1 &= \dfrac{e^{2 \mu_* L} - \cos{2 \mu_* L} }{2}, \quad 
K_2  = - \dfrac{\sin{2 \mu_* L}}{2}, \quad 
\mu_* = \dfrac{\mu^{1/4}}{\sqrt{2}}. 
\end{align*}
By virtue of \eqref{sol-ex} and \eqref{G-ex}, 
we see that the solution of \eqref{bd-prob-1} satisfies a priori estimate 
\begin{align} \label{a-priori}
\Wn{\vp}{4}{p} \leq C \Lp{f}{p}  
\end{align} 
for any $p \geq 1$. 
Using the a priori estimate \eqref{a-priori}, 
we show that the operator $A_0$ generates an analytic semigroup on $L^p(0,L)$. 
Moreover we can verify that $A_0 : h^{4+4\vs}_B([0,L]) \to h^{4\vs}_B([0,L])$ is an 
infinitesimal generator of an analytic semigroup on $h^{4\vs}_B([0,L])$, 
where $0 < \vs < 1/4$ (for example, see \cite{lunardi}). 
Here $h^\va_B([0,L])$ is a little H\"older space with boundary condition: 
\begin{align}
h^{\va}_B([0,L])= 
\begin{cases}
\left\{ u \in h^\va([0,L]) \mid u(0)=u(L)=u''(0)=u''(L)=0 \right\} \quad \text{if} \quad \va>2, \\
\left\{ u \in h^\va([0,L]) \mid u(0)=u(L)=0 \right\} \quad \text{if} \quad 0< \va < 2. 
\end{cases}
\end{align}
Since the equation in \eqref{d-sys} is a fourth order quasilinear parabolic equation, 
we shall prove a short time existence of \eqref{d-sys} as follows. 
Letting $B(d):= A(d) - A_0$, the system \eqref{d-sys} is written as 
\begin{align} \label{d-sys-2}
\begin{cases}
& \pd_t d  = A_0 d + B(d)d + \Phi(d), \\
& d(0,t) = d(L,t) = d''(0,t) = d''(L,t) =0, \\ 
& d(x,0)= d_0(x)= 0. 
\end{cases}
\end{align}
And then, we find a solution of \eqref{d-sys-2} for a short time by using contraction mapping principle. 
Indeed, making use of the maximal regularity property and continuous interpolation spaces, 
we see that there exists a unique classical solution of \eqref{d-sys-2}, i.e., 
\eqref{d-sys}, in the class $C([0,T]; h^{4+4\vs}_B([0,L])) \cap C^1([0,T];h^{4\vs}_B([0,L]))$, 
where $T>0$ is sufficiently small. 
And then we obtain the regularity by a standard bootstrap argument (see \cite{lunardi}). 
Then we obtain the following: 
\begin{lem} \label{cont-map}
Let $\Gm_0$ be a smooth curve satisfying \eqref{i-cond-cc}. 
Then there exists a constant $T > 0$ such that 
the problem \eqref{d-sys} has a unique smooth solution for $0 \leq t < T$. 
\end{lem}
Lemma \ref{cont-map} implies the existence of unique solution of \eqref{comp-sys} for a short time: 
\begin{thm} \label{loc-exe}
Let $\Gm_0(x)$ be a smooth curve satisfying \eqref{i-cond-cc}. 
Then there exist a constant $T > 0$ and a smooth curve $\gm(x,t)$ such that 
$\gm(x,t)$ is a unique classical solution of the problem \eqref{comp-sys} for $0 \leq t < T$. 
\end{thm}
\subsection{Long time existence}

Next we shall prove a long time existence of solution to \eqref{comp-sys}. 
Let us set 
\begin{align*}
F^\lm = 2 \pd^2_s \vk + \vk^3 - \lm^2 \vk. 
\end{align*}
Then the gradient flow \eqref{s-s-flow} is written as 
\begin{align*}
\pd_t \gm = - F^\lm \bn. 
\end{align*}
\begin{lem} \label{comu}
Under \eqref{s-s-flow}, the following commutation rule holds{\rm :} 
\begin{align*}
\pd_t \pd_s = \pd_s \pd_t - \vk F^\lm \pd_s. 
\end{align*}
\end{lem}
Lemma \ref{comu} gives us the following: 
\begin{lem} \label{k-flow}
Let $\gm(x,t)$ satisfy \eqref{s-s-flow}. 
Then the curvature $\vk(x,t)$ of $\gm(x,t)$ satisfies 
\begin{align} \label{k-eq}
\pd_t \vk & = - \pd^2_s F^\lm - \vk^2 F^\lm \\
 & = -2 \pd^4_s \vk -5 \vk^2 \pd^2_s \vk + \lm^2 \pd^2_s \vk 
      -6 \vk (\pd_s \vk)^2 - \vk^5 + \lm^2 \vk^3. \notag
\end{align}
Furthermore, the line element $ds$ of $\gm(x,t)$ satisfies 
\begin{align} \label{ds-pd}
\pd_t ds = \vk F^\lm ds = (2 \vk \pd^2_s \vk + \vk^4 - \lm^2 \vk^2) ds. 
\end{align}
\end{lem}
Here we introduce the following notation for a convenience. 
\begin{dfn} {\rm (\cite{b-m-n})}  \label{q-def}
We use the symbol $\mfq^r(\pd^l_s \vk)$ for a polynomial with constant coefficients 
such that each of its monomials is of the form 
\begin{align*}
\prod^{N}_{i=1} \pd^{j_{i}}_s \vk 
\quad \text{\rm with} \quad 0 \leq j_{i} \leq l 
\quad \text{\rm and} \quad N \geq 1
\end{align*}
with 
\begin{align*}
r= \sum^{N}_{i=1}(j_{i} +1). 
\end{align*}
\end{dfn}
By virtue of Lemmas \ref{comu} and \ref{k-flow}, we have 
\begin{lem} \label{l-high-deri}
For any $j \in \N$, the following formula holds{\rm :} 
\begin{align} \label{high-deri}
\pd_t \pd^j_s \vk 
 = -2 \pd^{j+4}_s \vk -5 \vk^2 \pd^{j+2}_s \vk + \lm^2 \pd^{j+2}_s \vk 
   + \lm^2 \mfq^{j+3}(\pd^j_s \vk) + \mfq^{j+5}(\pd^{j+1}_s \vk). 
\end{align}
\end{lem}
\begin{proof}
The case $j=0$ in \eqref{high-deri} has been already proved in Lemma \ref{k-flow}, where 
$\mfq^5(\pd_s \vk) = -6 \vk (\pd_s \vk)^2 -\vk^5$ and $\mfq^3(\vk)= \vk^3$. 
Next suppose that the formula \eqref{high-deri} holds for $j-1$. 
Then we have 
\begin{align*}
\pd_t \pd^j_s \vk 
& = \pd_s \pd_t \pd^{j-1}_s \vk - \vk F^\lm \pd^j_s \vk \\
& = \pd_s \left\{ -2 \pd^{j+3}_s \vk -5 \vk^2 \pd^{j+1}_s \vk + \lm^2 \pd^{j+1}_s \vk 
                   + \lm^2 \mfq^{j+2}(\pd^{j-1}_s \vk) + \mfq^{j+4}(\pd^{j}_s \vk) \right\} \\
& \qquad - \vk (2 \pd^2_s \vk + \vk^3 - \lm \vk^2 ) \pd^j_s \vk \\
& = -2 \pd^{j+4}_s \vk -5 \vk^2 \pd^{j+2}_s \vk + \lm^2 \pd^{j+2}_s \vk 
      + \lm^2 \mfq^{j+3}(\pd^{j}_s \vk) + \mfq^{j+5}(\pd^{j+1}_s \vk) . 
\end{align*}
\end{proof}

From the boundary condition of \eqref{comp-sys}, we see that the curvature $\vk$ satisfies the following: 
\begin{lem} \label{l-odd-k}
Let $\vk(x,t)$ be the curvature of $\gm(x,t)$ satisfying \eqref{comp-sys}. 
Then, for any $m \in \N$, it holds that 
\begin{align} \label{odd-k}
\partial^{2m}_s \vk(0,t)= \partial^{2m}_s \vk(L,t) =0. 
\end{align}
\end{lem}
\begin{proof}
First we show the case where $m= 1$, $2$. 
Differentiating the boundary condition $\gm(0,t)=(0,0)$ and $\gm(L,t)=(R,0)$ with respect to $t$, 
we have $\pd_t \gm(0,t)= \pd_t \gm(L,t)=0$. 
From $\vk(0,t)=\vk(L,t)=0$ and the equation \eqref{s-s-flow}, we see that 
$\pd^2_s \vk(0,t)=\pd^2_s \vk(L,t)=0$. Since $\partial_t \vk(0,t)= \partial_t \vk(L, t)=0$, 
the equation \eqref{k-eq} yields $\partial^4_s \vk(0,t)= \partial^4_s \vk(L,t)=0$.  

Next, suppose that $\partial^{2n}_s \vk(0,t)= \partial^{2n}_s \vk(L,t)=0$ holds 
for any natural numbers $0 \leq n \leq m$. Lemma \ref{l-high-deri} gives us 
\begin{align*}
\pd_t \pd^{2m-2}_s \vk 
& = -2 \pd^{2m+2}_s \vk -5 \vk^2 \pd^{2m}_s \vk + \lm^2 \pd^{2m}_s \vk 
      + \lm^2 \mfq^{2m+1}(\pd^{2m-2}_s \vk) + \mfq^{2m+3}(\pd^{2m-1}_s \vk) . 
\end{align*}
Since any monomials of $\mfq^{2m+1}(\pd^{2m-2}_s \vk)$ and $\mfq^{2m+3}(\pd^{2m-1}_s \vk)$ 
contain at least one of the terms $\partial^{2l}_s \vk$ $(l = 0$, $1$, $2$, $\cdots$, $m-1)$, 
we obtain $\partial^{2m+2}_s \vk(0,t)= \partial^{2m+2}_s \vk(L,t)=0$. 
\end{proof}
Let us define $L^p$ norm with respect to the arc length parameter of $\gm$. 
For a function $f(s)$ defined on $\gm$, we write  
\begin{align*}
\Lns{f}{p} &= \left\{ \int_\gm \av{f(s)}^p \, ds \right\}^\frac{1}{p}, \\
\Lns{f}{\infty} &= \su{s \in [0, \mL(\gm)]}{\av{f(s)}}, 
\end{align*}
where $\mL(\gm)$ denotes the length of $\gm$. 

In the following, we make use of the following interpolation inequalities: 
\begin{lem} \label{l-inter-p}
Let $\gm(x,t)$ be a solution of \eqref{comp-sys}. 
Let $u(x,t)$ be a function defined on $\gm$ and satisfy 
\begin{align*}
\pd^{2m}_s u(0,t) = \pd^{2m}_s u(L,t)= 0  
\end{align*}
for any $m \in \N$. Then, for integers $0 \leq p < q < r$, it holds that 
\begin{align} \label{inter-p}
\Lns{\pd^q_s u}{2} \leq \Lns{\pd^p_s u}{2}^{\frac{r-q}{r-p}} \Lns{\partial^r_s u}{2}^{\frac{q-p}{r-p}}. 
\end{align}
Moreover, for integers $0 \leq p \leq q < r$, it holds that 
\begin{align} \label{inter-p-2}
\Lns{\pd^q_s u}{\infty} \leq \sqrt{2} \Lns{\pd^p_s u}{2}^{\frac{2(r-q)-1}{2(r-p)}} 
                                         \Lns{\pd^r_s u}{2}^{\frac{2(q-p)+1}{2(r-p)}}. 
\end{align}
\end{lem}
\begin{proof}
Making use of Lemma \ref{l-odd-k}, for any positive integer $n$, we have 
\begin{align*}
\Lns{\pd^n_s u}{2}^2 = \int_\gm (\pd^n_s u)^2 \, ds 
  = - \int_\gm \pd^{n-1}_s u \cdot \pd^{n+1}_s u \, ds 
  \leq \Lns{\pd^{n-1}_s u}{2} \Lns{\pd^{n+1}_s u}{2}. 
\end{align*}
This implies that $\log{\Lns{\partial^n_s u}{2}}$ is convex with respect to $n>0$. 
Thus we obtain the inequality \eqref{inter-p}. 

Next we turn to \eqref{inter-p-2}. 
By Lemma \ref{l-odd-k}, we see that $\pd^{2m}_s u(0)= \pd^{2m}_s u(L)=0$ for any $m \in \N$. 
Thus the intermediate theorem implies that there exists at least one point $0 < \xi < L$ such that 
$\pd^{2m+1}_s u(\xi)=0$ for each $m \in \N$. 
Hence, for each non-negative integer $n$, there exists a point $0 < \xi < L$ such that $\pd^n_s u(\xi)=0$. 
Then we have 
\begin{align*}
\left( \pd^n_s u(s) \right)^2 
 = \int^s_\xi \left\{ \left( \pd^n_s u(\tau) \right)^2 \right\}' \, d \tau 
 \leq 2 \Lns{\pd^n_s u}{2} \Lns{\pd^{n+1}_s u}{2}. 
\end{align*}
Hence we get  
\begin{align} \label{i-2-2}
\Lns{\partial^n_s u}{\infty} \leq 
 \sqrt{2} \Lns{\pd^n_s u}{2}^{\frac{1}{2}} \Lns{\pd^{n+1}_s u}{2}^{\frac{1}{2}}. 
\end{align}
Combining \eqref{inter-p} with \eqref{i-2-2}, we obtain \eqref{inter-p-2}. 
\end{proof}
By virtue of Lemma \ref{l-odd-k}, we are able to apply Lemma \ref{l-inter-p} to $\vk(x,t)$ for any $n \in \N$. 
Making use of boundedness of the energy functional at $\gm=\Gm_0$, we derive an estimate for $\Lns{\vk}{2}$: 
\begin{lem} \label{k-est-lm-1}
The following estimate holds{\rm :} 
\begin{align} \label{k-est-1}
\Lns{\vk}{2}^2 \leq \Ln{k_0}{0,L}^2 + \lm^2 \left( \mL(\Gm_0) - R \right). 
\end{align}
\end{lem}
\begin{proof}
Since the equation in \eqref{comp-sys} is the steepest descent flow for 
$E(\gm) = \Lns{\vk}{2}^2 + \lm^2 \mL(\gm)$, we have 
\begin{align*}
\Lns{\vk}{2}^2 + \lm^2 \mL(\gm) 
 \leq \Ln{k_0}{0,L}^2 + \lm^2 \mL(\Gm_0). 
\end{align*}
Clearly it holds that $\mL(\gm) \geq R$. Therefore we obtain \eqref{k-est-1}. 
\end{proof}
In order to use the energy method, we prepare the following: 
\begin{lem} \label{l-deri-high}
For any $j \in \N$, it holds that 
\begin{align} \label{deri-high}
\dfrac{d}{dt} \Lns{\pd^j_s \vk}{2}^2  
 & = -2 \Lns{\pd^{j+2}_s \vk}{2}^2 -2 \lm^2 \Lns{\pd^{j+1}_s \vk}{2}^2 \\
 & \qquad + \lm^2 \int_\gm \mfq^{2j+4}(\pd^{j}_s \vk)\, ds + \int_\gm \mfq^{2j+6}(\pd^{j+1}_s \vk). \notag
\end{align}
\end{lem}
\begin{proof}
By virtue of Lemma \ref{l-high-deri}, we have 
\begin{align*}
\dfrac{d}{dt} \Lns{\pd^j_s \vk}{2}^2 
& = \int_\gm 2 \pd^j_s \vk \pd_t \pd^j_s \vk \, ds + \int_\gm (\pd^j_s \vk)^2 \vk F^\lm \, ds \\
& = \int_\gm 2 \pd^j_s \vk \left\{ -2 \pd^{j+4}_s \vk -5 \vk^2 \pd^{j+2}_s \vk + \lm^2 \pd^{j+2}_s \vk 
             + \lm^2 \mfq^{j+3}(\pd^j_s \vk) + \mfq^{j+5}(\pd^{j+1}_s \vk) \right\} \, ds \\
& \qquad + \int_\gm \vk \pd^j_s \vk (2 \pd^2_s \vk + \vk^3 - \lm \vk^2 ) \, ds. 
\end{align*}
By integrating by parts, we get 
\begin{align*}
\int_\gm \vk^2 \pd^j_s \vk \pd^{j+2}_s \vk \, ds 
 = - \int_\gm \left\{ 2 \vk \pd_s \vk \pd^j_s \vk \pd^{j+1}_s \vk + \vk^2 (\pd^{j+1}_s \vk)^2 \right\} \, ds. 
\end{align*}
Consequently we obtain \eqref{deri-high}. 
\end{proof}
Using Lemmas \ref{k-est-lm-1} and \ref{l-deri-high}, 
we derive the estimate for the derivative of $\| \pd^j_s \vk \|_{L^2_s}^2$ with respect to $t$. 
\begin{lem} \label{high-energy}
For any $j \in \N$,  we have 
\begin{align*}
\dfrac{d}{dt} \Lns{\pd^j_s \vk}{2}^2  
 \leq C \Lns{\vk}{2}^{4j+6} + C \Lns{\vk}{2}^{4j+10} . 
\end{align*}
\end{lem}
\begin{proof}
By Lemma \ref{l-deri-high}, we shall estimate the right hand side of \eqref{deri-high}. 
First we focus on the term $\int_\gm \mfq^{2j+4} (\pd^j_s \vk)\, ds$. 
By Definition \ref{q-def}, we have 
\begin{align*}
\mfq^{2j+4}(\pd^j_s \vk) = \sum_m \prod^{N_m}_{l=1} \pd^{c_{ml}}_s \vk
\end{align*}
with all the $c_{ml}$ less than or equal to $j$ and 
\begin{align*}
\sum^{N_m}_{l=1}(c_{ml} +1) = 2j+4
\end{align*}
for every $m$. 
Hence we have 
\begin{align*}
\av{\mfq^{2j+4}(\pd^j_s \vk)} \leq \sum_m \prod^{N_m}_{l=1} \av{\pd^{c_{ml}}_s \vk}. 
\end{align*}
Setting 
\begin{align*}
Q_m = \prod^{N_m}_{l=1} \av{\pd^{c_{ml}}_s \vk}, 
\end{align*}
it holds that  
\begin{align*}
\int_\gm \av{\mfq^{2j+4}(\pd^j_s \vk)} \, ds 
 \leq \sum_m \int_\gm Q_m \, ds. 
\end{align*}
We now estimate any term $Q_m$ by Lemma \ref{l-inter-p}. 
After collecting derivatives of the same order in $Q_m$, we can write 
\begin{align} \label{q-cond}
Q_m = \prod^{j}_{i=0} \av{\pd^i_s \vk}^{\va_{mi}} 
 \qquad \text{with} \quad 
\sum^j_{i=0} \va_{mi} (i+1) = 2j +4. 
\end{align}
Then 
\begin{align*}
\int_\gm Q_m \, ds 
 = \int_\gm \prod^{j}_{i=0} \av{\pd^i_s \vk}^{\va_{mi}} \, ds 
 \leq \prod^{j}_{i=0} \left( \int_\gm \av{\pd^i_s \vk}^{\va_{mi}} \, ds \right)^{1/\lm_i}
 \leq \prod^{j}_{i=0} \Lns{\pd^i_s \vk}{\va_{mi \lm_i}}^{\va_{mi}}, 
\end{align*}
where the value $\lm_i$ are chosen as follows: 
$\lm_i=0$ if $\va_{mi}=0$ (in this case the corresponding term is not present in the product) 
and $\lm_i = (2j+4)/(\va_{mi}(i+1))$ if $\va_{mi} \neq 0$.  
Clearly, $\va_{mi} \lm_i = \frac{2j+4}{i+1} \geq \frac{2j+4}{j+1} > 2$ and 
by the condition \eqref{q-cond}, 
\begin{align*}
\sum^{j}_{i=0, \lm_i \neq 0} \frac{1}{\lm_i} 
 = \sum^{j}_{i=0, \lm_i \neq 0} \frac{\va_{mi}(i+1)}{2j+4} =1. 
\end{align*}
Let $k_i = \va_{mi} \lm_i -2$. The fact $\va_{mi} \lm_i > 2$ implies $k_i >0$. 
Then we have 
\begin{align*}
\Lns{\pd^i_s \vk}{\va_{mi \lm_i}} 
& = \Lns{\pd^i_s \vk}{\infty}^{k_i} \Lns{\pd^i_s \vk}{2}^2, \\
\Lns{\pd^i_s \vk}{\infty}^{k_i} 
 & \leq 2^{\frac{k_i}{2}} \Lns{\pd^{j+1}_s \vk}{2}^{\frac{2i +1}{2j+2}k_i} 
                        \Lns{\vk}{2}^{\frac{2j+1-2i}{2j+2}k_i}, \\
\Lns{\pd^i_s \vk}{2}^2  
 & \leq \Lns{\pd^{j+1}_s \vk}{2}^{\frac{2i}{j+1}} 
                        \Lns{\vk}{2}^{\frac{2j+2-2i}{j+1}}. 
\end{align*}
These imply  
\begin{align*}
\Lns{\pd^i_s \vk}{\va_{mi} \lm_i} 
& \leq 2^{\frac{k_i}{2}} \Lns{\pd^{j+1}_s \vk}{2}^{\sigma_{mi}} \Lns{\vk}{2}^{1-\sigma_{mi}} 
\end{align*}
with 
\begin{align*}
\sigma_{mi} = \frac{i+ \frac{1}{2} -\frac{1}{\va_{mi} \lm_i} }{j+1}. 
\end{align*}
Multiplying together all the estimates, 
\begin{align} \label{Q-ineq-1}
\int_\gm Q_m \, ds 
& \leq \prod^{j}_{i=0} 2^{\frac{k_i}{2}} \Lns{\pd^{j+1}_s \vk}{2}^{\va_{mi} 
                    \sigma_{mi}} \Lns{\vk}{2}^{\va_{mi}(1-\sigma_{mi})}\\
& \leq C \Lns{\pd^{j+1}_s \vk}{2}^{\sum^{j}_{i=0} \va_{mi} \sigma_{mi}} 
          \Lns{\vk}{2}^{\sum^{j}_{i=0} \va_{mi}(1-\sigma_{mi})}. \notag
\end{align}
Then we compute 
\begin{align*}
\sum^{j}_{i=0} \va_{mi} \sigma_{mi} 
 = \sum^{j}_{i=0} \frac{\va_{mi}(i+ \frac{1}{2}) - \frac{1}{\lm_i}}{j+1} 
 = \frac{\sum^{j}_{i=0} \va_{mi}(i+ \frac{1}{2}) -1 }{j+1}
\end{align*}
and using again the rescaling condition in \eqref{q-cond}, 
\begin{align*}
\sum^{j}_{i=0} \va_{mi} \sigma_{mi} 
& = \frac{\sum^{j}_{i=0} \va_{mi}(i+ 1) -\frac{1}{2} \sum^{j}_{i=0} \va_{mi} -1 }{j+1} \\
& = \frac{2j+4 -\frac{1}{2} \sum^{j}_{i=0} \va_{mi} -1 }{j+1} 
  = \frac{4j+6 - \sum^{j}_{i=0} \va_{mi} }{2(j+1)}. 
\end{align*}
Since 
\begin{align*}
\sum^{j}_{i=0} \va_{mi} 
 \geq \sum^{j}_{i=0} \va_{mi} \frac{i+1}{j+1}
 = \frac{2j+4}{j+1}, 
\end{align*}
we get 
\begin{align*}
\sum^{j}_{i=0} \va_{mi} \sigma_{mi} 
 \leq \frac{2j^2 +4j +1}{(j+1)^2}
 = 2- \frac{1}{(j+1)^2} < 2. 
\end{align*}
Hence, we can apply the Young inequality to the product in the last term of inequality \eqref{Q-ineq-1}, 
in order to get the exponent $2$ on the first quantity, that is, 
\begin{align*}
\int_\gm Q_m \, ds 
 \leq \frac{\vd_m}{2} \Lns{\pd^{j+1}_{s} \vk}{2}^2 + C_m \Lns{\vk}{2}^\vb 
 \leq \vd_m \Lns{\pd^{j+1}_{s} \vk}{2}^2 \, ds + C_m \Lns{\vk}{2}^\vb, 
\end{align*}
for arbitrarily small $\vd_m > 0$ and some constant $C_m >0$. 
The exponent $\vb$ is given by 
\begin{align*}
\vb & = \sum^{j}_{i=0} \va_{mi} (1-\sigma_{mi}) \frac{1}{1- \frac{\sum^{j}_{i=0} \va_{mi} \sigma_{mi}}{2} } 
      = \frac{ 2 \sum^{j}_{i=0} \va_{mi}(1- \sigma_{mi}) }{2- \sum^{j}_{i=0} \va_{mi} \sigma_{mi}} \\
& = \frac{2 \sum^{j}_{i=0} \va_{mi} - \frac{4j+6 - \sum^{j}_{i=0} \va_{mi} }{j+1} } 
         {2- \frac{4j+6 - \sum^{j}_{i=0} \va_{mi} }{2(j+1)}} 
  = 2 \frac{2(j+1) \sum^{j}_{i=0} \va_{mi} -4j-6+ \sum^{j}_{i=0} \va_{mi} }
           {4j+4-4j-6+ \sum^{j}_{i=0} \va_{mi}} \\
& = 2 \frac{(2j+3) \sum^{j}_{i=0} \va_{mi} -2(2j+3)}
           {\sum^{j}_{i=0} \va_{mi} -2}
  = 2(2j+3). 
\end{align*}
Therefore we conclude 
\begin{align*}
\int_\gm Q_m \, ds 
 \leq \vd_m \Lns{\pd^{j+1}_s \vk}{2}^2 
       + C_m \Lns{\vk}{2}^{4j+6}. 
\end{align*}
Repeating this argument for all the $Q_m$ and choosing suitable $\vd_m$ whose sum over $m$ is less than one, 
we conclude that there exists a constant $C$ depending only on $j \in \N$ such that 
\begin{align*}
\int_\gm \mfq^{2j+4}(\pd^j_s \vk) \, ds 
 \leq \Lns{\pd^{j+1}_s \vk}{2}^2 
       + C \Lns{\vk}{2}^{4j+6}. 
\end{align*}
Reasoning similarly for the term $\mfq^{2j+6}(\pd^{j+1}_s \vk)$, we obtain 
\begin{align*}
\int_\gm \mfq^{2j+6}(\pd^{j+1}_s \vk) \, ds 
 \leq \Lns{\pd^{j+2}_s \vk}{2}^2 
       + C \Lns{\vk}{2}^{4j+10}. 
\end{align*}
Hence, from \eqref{deri-high}, we get 
\begin{align*}
\pd_t \Lns{\pd^j_s \vk}{2}^2  
& = -2 \Lns{\pd^{j+2}_s \vk}{2}^2 -2 \lm^2 \Lns{\pd^{j+1}_s \vk}{2}^2 \\
 & \qquad + \lm^2 \int_\gm \mfq^{2j+4}(\pd^{j}_s \vk)\, ds + \int_\gm \mfq^{2j+6}(\pd^{j+1}_s \vk) \\
& \leq - \lm^2 \Lns{\pd^{j+1}_s \vk}{2}^2  
       + C \Lns{\vk}{2}^{4j+6}  - \Lns{\pd^{j+2}_s \vk}{2}^2 + C \ve \Lns{\vk}{2}^{4j+10} \\
& \leq C \Lns{\vk}{2}^{4j+6} 
        + C \Lns{\vk}{2}^{4j+10} ,   
\end{align*} 
where $C$ depends only on $j$. 
\end{proof}
Next we estimate the local length of $\gm(x,t)$. 
\begin{lem} \label{ele-bd-lm-1}
Let $\gm(x,t)$ be a solution of \eqref{comp-sys} for $0 \leq t < T$. 
Then there exist positive constants $C_1$ and $C_2$ such that the inequalities 
\begin{align} 
& \dfrac{1}{C_1(\Gm_0,T)} \leq \av{\pd_x \gm(x,t)} \leq C_1(\Gm_0, T), \label{ele-bd-1-1} \\
& \av{\pd^m_x \av{\pd_x \gm(x,t)}} \leq C_2(\Gm_0,T) \label{ele-bd-1-2}
\end{align}
hold for any $(x,t) \in [0, L] \times [0, T]$ and integer $m \geq 1$. 
\end{lem}
\begin{proof} 
First we prove \eqref{ele-bd-1-1}. Since 
\begin{align*}
\pd_x \pd_t \gm 
& = \pd_x \left( -2 \pd^2_s \vk - {\vk}^3 + \lm^2 \vk \right) \bn 
     + \left( -2 \pd^2_s \vk - {\vk}^3 + \lm^2 \vk \right) \pd_x \bn, 
\end{align*}
and $\pd_x \bn = \av{\pd_x \gm} \pd_s \bn = - \av{\pd_x \gm} \vk \pd_s \gm$, 
we have 
\begin{align} \label{eq-le-1-1}
\pd_t \av{\pd_x \gm} 
 = \dfrac{\pd_x \gm \cdot \pd_x \pd_t \gm}{\av{\pd_x \gm}} 
 = - \vk \left( -2 \pd^2_s \vk - {\vk}^3 + \lm^2 \vk \right) \av{\pd_x \gm}. 
\end{align}
Thus $\av{\pd_x \gm}$ satisfies the initial value problem 
\begin{align} \label{od-eq-1-1}
\begin{cases}
\dfrac{d u}{dt} = F(\vk) u, \\
u(0) = 1, 
\end{cases}
\end{align} 
where 
\begin{align*}
F(\vk) = - \vk \left( -2 \pd^2_s \vk - {\vk}^3 + \lm^2 \vk \right). 
\end{align*}
Since Lemmas \ref{k-est-lm-1} and \ref{high-energy} implies that 
there exists a constant $C$ such that 
\begin{align*}
\av{F(\vk)} \leq C 
\end{align*}
for any $(x,t) \in [0,L] \times [0,T]$. 
Hence, for any $(x,t) \in [0,L] \times [0,T]$, we have 
\begin{align*}
e^{-C_1 T} \leq \av{\pd_x \gm(x,t)} \leq e^{C_1 T}. 
\end{align*}
Next we turn to the proof of \eqref{ele-bd-1-2}. 
Here we have 
\begin{align} \label{bd-1-1}
\pd^m_x F(\vk) - \av{\pd_x \gm}^m \pd^m_s F(\vk) 
 = P(\av{\pd_x \gm}, \cdots, \pd^{m-1}_x \av{\pd_x \gm}, F(\vk), \cdots, \pd^{m-1}_s F(\vk)). 
\end{align}
Suppose that there exist constants $C_j(T, \Gm_0)$ such that 
\begin{align*}
\sup_{(x,t) \in [0,L] \times [0,T]} \av{\pd^j_x \av{\pd_x \gm} } \leq C_j (T,\Gm_0) 
\end{align*}
for any $0 \leq j \leq m-1$. Then \eqref{bd-1-1} implies 
\begin{align*}
\av{\pd^m_x F(\vk)} < C.  
\end{align*}
Differentiating the equation \eqref{eq-le-1-1} with respect to $x$, we have 
\begin{align*}
\pd_t \pd^m_x \av{\pd_x \gm} 
  = F(\vk) \pd^m_x \av{\pd_x \gm} 
     + \sum^{m}_{j=1} \mbox{}_m C_j \pd^j_x F(\vk) \pd^{m-j}_x \av{\pd_x \gm}. 
\end{align*}
Thus $\pd^m_x \av{\pd_x \gm}$ satisfies  
\begin{align} \label{od-eq-1-2}
\begin{cases}
& \pd_t v = F(\vk) v + G, \\
& v(0) = 0. 
\end{cases}
\end{align}
We can check that there exists a constant $C_2(T,\Gm_0)$ such that $\av{v} \leq C_2$. 
This gives us the conclusion of Lemma \ref{ele-bd-lm-1}. 
\end{proof}
Then we prove that the system \eqref{comp-sys} has a unique global solution in time. 
\begin{thm} \label{comp-gm-global}
Let $\Gm_0$ be a smooth planar curve satisfying the condition \eqref{i-cond-cc}. 
Then there exists a unique classical solution of \eqref{comp-sys} for any time $t>0$. 
\end{thm}
\begin{proof}
Suppose not, then there exists a positive constant $\tilde{T}$ such that 
$\gm(x,t)$ does not extend smoothly beyond $\tilde{T}$. 
It follows from Lemmas \ref{k-est-lm-1} and \ref{high-energy} that 
\begin{align*}
\Lns{\pd^m_s \vk}{2}^2 
 \leq \Ln{\pd^m_x k_0}{0,L}^2 + C \tilde{T} 
\end{align*}
holds for any $0 \leq t \leq \tilde{T}$ and $m \in \N$. This yields that there exists 
a constant $C$ such that  
\begin{align} \label{gm-high-est}
\Lns{\pd^m_s \gm}{2} \leq C 
\end{align}
for $t \in [0,\tilde{T}]$. We have already known 
\begin{align} \label{gm-bd-x} 
\pd^m_x \gm - \av{\pd_x \gm}^m \pd^m_s \gm 
 = P(\av{\pd_x \gm}, \cdots, \pd^{m-1}_x \av{\pd_x \gm}, \gm, \cdots, \pd^{m-1}_s \gm). 
\end{align}
By virtue of \eqref{gm-high-est}, \eqref{gm-bd-x}, and Lemma \ref{ele-bd-lm-1}, 
we see that there exists a constant $C$ such that 
\begin{align*}
\Ln{\pd^m_x \gm}{0,L} \leq C 
\end{align*}
for any $t \in [0,\tilde{T}]$ and $m \in \N$. 
Then $\gm(x,t)$ extends smoothly beyond $\tilde{T}$ by Theorem \ref{loc-exe}. 
This is a contradiction. We complete the proof. 
\end{proof}

\subsection{Convergence to a stationary solution} \label{convergence}
Finally we shall prove that the solution $\gm(x,t)$ converges to a stationary solution 
as $t \to \infty$. 
For this purpose, we rewrite the equation \eqref{s-s-flow} in terms of $\gm$ as follows:  
\begin{align} \label{gf}
\pd_t \gm = - 2 \pd^4_s \gm +\left(\lm^2 - 3 \av{\pd^2_s \gm}^2 \right) \pd^2_s \gm  
                     - 3 \pd_s \left( \av{\pd^2_s \gm}^2 \right) \pd_s \gm .
\end{align}
Since the arc length parameter $s$ depends on $t$, the following rules hold: 
\begin{align} 
\pd_t \pd_s & = \pd_s \pd_t - \Gl \pd_s, \label{st-ts} \\
\pd_t ds &= \Gl ds, \label{pt-ds-G}
\end{align}
where $G^\lm = \pd_s \pd_t \gm \cdot \pd_s \gm$. 
In previous section, we prove that the initial-boundary value problem for \eqref{gf} 
has a unique classical solution $\gm(x,t)$ for any $t>0$. 
The solution $\gm$ has the following property: 
\begin{lem} \label{l-gm-bd-m}
Let $\gm(x,t)$ be the solution of \eqref{comp-sys}. 
Then, for any positive integer $m$, it holds that 
\begin{align} \label{gm-bd-m}
\pd^{2m}_s \gm(0,t) = \pd^{2m}_s \gm(L,t) = 0. 
\end{align}
\end{lem}
\begin{proof}
First we prove that the relation 
\begin{align} \label{gm-rela}
\pd^n_s \gm = \mfq^{n-1}(\pd^{n-2}_s \vk) \bn + \mfq^{n-1}(\pd^{n-2}_s \vk) \pd_s \gm 
\end{align}
holds for any integers $n \geq 2$. Since $\pd^2_s \gm = \vk \bn$, we see that \eqref{gm-rela} holds for $n=2$. 
Suppose that \eqref{gm-rela} holds for any integers $2 \leq n \leq k$, where $k>2$ is some integer. 
Then we have 
\begin{align*}
\pd^{k+1}_s \gm 
 &= \pd^s \{ \mfq^{k-1}(\pd^{k-2}_s \vk) \} \bn + \mfq^{k-1}_s(\pd^{k-2}_s \vk) \pd_s \bn 
    + \pd^s \{ \mfq^{k-1}(\pd^{k-2}_s \vk) \} \pd_s \gm + \mfq^{k-1}_s(\pd^{k-2}_s \vk) \pd^2_s \gm \\
 &= \left\{ \pd^s \{ \mfq^{k-1}(\pd^{k-2}_s \vk) \} + \vk \mfq^{k-1}_s(\pd^{k-2}_s \vk) \right\} \bn 
    + \left\{ \pd^s \{ \mfq^{k-1}(\pd^{k-2}_s \vk) \} 
     - \vk \mfq^{k-1}_s(\pd^{k-2}_s \vk) \right\} \pd_s \gm \\
 &= \mfq^k(\pd^{k-1}_s \vk) \bn + \mfq^k(\pd^{k-1}_s \vk) \pd_s \gm. 
\end{align*}
This implies \eqref{gm-rela}. 
Then, along the same line as in the proof of Lemma \ref{l-odd-k}, 
we obtain the conclusion. 
\end{proof}
By virtue of Lemma \ref{l-gm-bd-m}, we can apply Lemma \ref{l-inter-p}, i.e., interpolation inequalities, 
to $\pd^2_s \gm$. 
Using the interpolation inequalities, we first prove the following estimate: 
\begin{lem} \label{l-est-deri-4}
There exist positive constants $C_1$ and $C_2$ depending only on $\lm$ such that 
\begin{align*}
\Lns{\pd^4_s \gm}{2} \leq \Lns{\pd_t \gm}{2} + C_1 \Lns{\pd^2_s \gm}{2}^{5} + C_2 \Lns{\pd^2_s \gm}{2}. 
\end{align*}
\end{lem}
\begin{proof}
From \eqref{gf}, we have 
\begin{align*}
\Lns{\pd_t \gm}{2}^2 
& = \int_\gm \av{- 2 \pd^4_s \gm +\left(\lm^2 - 3 \av{\pd^2_s \gm}^2 \right) \pd^2_s \gm  
                     - 3 \pd_s \left( \av{\pd^2_s \gm}^2 \right) \pd_s \gm}^2\, ds \\
& \geq 2 \Lns{\pd^4_s \gm}{2}^2 
        - 2 \Lns{\left(\lm^2 - 3 \av{\pd^2_s \gm}^2 \right) \pd^2_s \gm}{2}^2 
          - 2 \Lns{3 \pd_s \left( \av{\pd^2_s \gm}^2 \right) \pd_s \gm }{2}^2. 
\end{align*}
It follows from interpolation inequalities that 
\begin{align*}
\Lns{\left(\lm^2 - 3 \av{\pd^2_s \gm}^2 \right) \pd^2_s \gm}{2}^2 
& \leq 2 \lm^4 \Lns{\pd^2_s \gm}{2}^2 + 18 \Lns{\av{\pd^2_s \gm}^2 \pd^2_s \gm}{2}^2 \\ 
& \leq 2 \lm^4 \Lns{\pd^2_s \gm}{2}^2 + 18 \Lns{\pd^2_s \gm}{\infty}^{4} \Lns{\pd^2_s \gm}{2}^2 \\
& \leq 2 \lm^4 \Lns{\pd^2_s \gm}{2}^2 + 72 \Lns{\pd^4_s \gm}{2} \Lns{\pd^2_s \gm}{2}^5 \\
& \leq \ve \Lns{\pd^4_s \gm}{2}^2 + C(1/\ve) \Lns{\pd^2_s \gm}{2}^{10} + 2 \lm^4 \Lns{\pd^2_s \gm}{2}^2 . 
\end{align*}
Similarly we have 
\begin{align*}
\Lns{\pd_s \left( \av{\pd^2_s \gm}^2 \right) \pd_s \gm }{2}^2 
& \leq 4 \Lns{\pd^2_s \gm}{\infty}^2 \Lns{\pd^3_s \gm}{2}^2 \\
& \leq 4 \sqrt{2} \Lns{\pd^4_s \gm}{2}^{3/2} \Lns{\pd^2_s \gm}{2}^{5/2} 
  \leq \ve \Lns{\pd^4_s \gm}{2}^2 + C(1/\ve) \Lns{\pd^2_s \gm}{2}^{10}. 
\end{align*}
Letting $\ve= \dfrac{1}{4}$, we obtain 
\begin{align*}
\Lns{\pd^4_s \gm}{2}^2 
 \leq \Lns{\pd_t \gm}{2}^2 + C_1 \Lns{\pd^2_s \gm}{2}^{10} + C_2 \Lns{\pd^2_s \gm}{2}^2. 
\end{align*}
\end{proof}
In order to derive the estimate of $\Lns{\pd^n_s \gm}{2}$ for $n \geq 5$, 
we prepare the following:  
\begin{lem} \label{commu-2}
For any $n \in \N$, it holds that  
\begin{align*}
\pd_t \pd^n_s \gm = \pd^n_s \pd_t \gm - \sum^{n-1}_{j=0}\pd^j_s(\Gl \pd^{n-j}_s \gm). 
\end{align*}
\end{lem}
Using Lemma \ref{commu-2}, we prove the estimate of $\Lns{\pd^{n+4}_s \gm}{2}$ for any $n \in \N$: 
\begin{lem} \label{l-est-deri-n}
For any $n \in \N$, the following estimate holds{\rm :}  
\begin{align} \label{est-deri-n}
\Lns{\pd^{n+4}_s \gm}{2} 
 \leq \Lns{\pd^{n}_s \pd_t \gm}{2} + C \Lns{\pd^2_s \gm}{2}^{2n+5} + C \Lns{\pd^2_s \gm}{2}. 
\end{align}
\end{lem}
\begin{proof}
We have already proved the case $n=0$ by Lemma \ref{l-est-deri-4}. 
Let $n \geq 1$ fix arbitrarily. 
By \eqref{gf}, we have 
\begin{align*}
\Lns{\pd^{n}_s \pd_t \gm}{2}^2 
& = \Lns{-2 \pd^{n+4}_s \gm 
     + \pd^{n}_s \left\{ \left( \lm^2 - 3 \av{\pd^2_s \gm}^2 \right) \pd^2_s \gm \right\}
      -3 \pd^{n}_s \left\{ \pd_s \left( \av{\pd^2_s \gm}^2 \right) \pd_s \gm \right\}  }{2}^2 \\
& \geq 2 \Lns{\pd^{n+4}_s \gm}{2}^2 
      -2 \Lns{\pd^{n}_s \left\{ \left( \lm^2 - 3 \av{\pd^2_s \gm}^2 \right) \pd^2_s \gm \right\} }{2}^2 
  -2 \Lns{3 \pd^{n}_s \left\{ \pd_s \left( \av{\pd^2_s \gm}^2 \right) \pd_s \gm \right\} }{2}^2 \\
&:= 2 \Lns{\pd^{n+4}_s \gm}{2}^2 -2 I_1 -2 I_2.  
\end{align*}
Concerning $I_1$, first we have 
\begin{align*}
\Lns{\pd^n_s \pd^2_s \gm}{2}^2 
 \leq C \Lns{\pd^{n+4}_s \gm}{2}^{\frac{2n}{n+2}} \Lns{\pd^2_s \gm}{2}^{\frac{4}{n+2}} 
 \leq \ve \Lns{\pd^{n+4}_s \gm}{2}^2 + C \Lns{\pd^2_s \gm}{2}^2 . 
\end{align*}
Furthermore since 
\begin{align*}
\pd^n_s \left( \av{\pd^2_s \gm}^2 \pd^2_s \gm \right) 
 = 2 \sum^{n}_{j=0} \mbox{}_n C_j \pd^{n-j+2}_s \gm \sum^{j}_{k=0} \mbox{}_j C_k \pd^{k+2}_s \gm \cdot \pd^{j-k+2}_s \gm , 
\end{align*}
it follows from interpolation inequalities that 
\begin{align*}
\Lns{\pd^n_s \left( \av{\pd^2_s \gm}^2 \pd^2_s \gm \right)}{2}^2 
& \leq C \sum^{n}_{j=0} \sum^{j}_{k=0} \Lns{\pd^{n-j+2}_s \gm (\pd^{k+2}_s \gm \cdot \pd^{j-k+2}_s \gm)}{2}^2 \\
& \leq C \sum^{n}_{j=0} \sum^{j}_{k=0}\Lns{\pd^{n-j+2}_s \gm}{2}^2 
                                       \Lns{\pd^{k+2}_s \gm}{\infty}^2 \Lns{\pd^{j-k+2}_s \gm}{\infty}^2 \\
& \leq C \Lns{\pd^{n+4}_s \gm}{2}^{\frac{2n+2}{n+2}} \Lns{\pd^2_s \gm}{2}^{\frac{4n+10}{n+2}} 
  \leq \ve \Lns{\pd^{n+4}_s \gm}{2}^2 + C \Lns{\pd^2_s \gm}{2}^{4n+10}. 
\end{align*}
Thus we have 
\begin{align*}
I_1 \leq (\lm^2 +1) \ve \Lns{\pd^{n+4}_s \gm}{2}^2 + C \Lns{\pd^2_s \gm}{2}^{4n+10} + C \Lns{\pd^2_s \gm}{2}^2. 
\end{align*}
Next we estimate $I_2$. Since 
\begin{align*}
\pd^n_s \left\{ \pd_s \left( \av{\pd^2_s \gm}^2 \right) \pd_s \gm \right\} 
 = 2 \sum^{n}_{j=0} \mbox{}_n C_j \pd^{n-j+1}_s \gm 
     \sum^{j}_{k=0} \mbox{}_j C_k \pd^{k+2}_s \gm \cdot \pd^{j-k+3}_s \gm, 
\end{align*}
we obtain 
\begin{align*}
I_2 
& \leq C \sum^{n}_{j=0} \sum^{j}_{k=j} \Lns{\pd^{n-j+1}_s \gm (\pd^{k+2}_s \gm \cdot \pd^{j-k+3}_s \gm)}{2}^2 \\
& \leq C \sum^{n}_{j=0} \sum^{j}_{k=j}
       \Lns{\pd^{n-j+1}_s \gm}{2}^2 \Lns{\pd^{k+2}_s \gm}{\infty}^2 \Lns{\pd^{j-k+3}_s \gm}{\infty}^2 \\
& \leq C \Lns{\pd^{n+4}_s \gm}{2}^{\frac{2n+2}{n+2}} \Lns{\pd^2_s \gm}{2}^{\frac{4n+10}{n+2}} 
  \leq \ve \Lns{\pd^{n+4}_s \gm}{2}^2 + C \Lns{\pd^2_s \gm}{2}^{4n+10}. 
\end{align*}
Letting $\ve>0$ sufficiently small, we complete the proof. 
\end{proof}
By virtue of Lemma \ref{commu-2}, we show that $\pd_t \gm$ satisfies the 
similar property to Lemma \ref{l-gm-bd-m}. 
\begin{lem} \label{gm-t-bdc}
Let $\gm(x,t)$ be a solution of \eqref{comp-sys}. Then it holds that 
\begin{align}
\pd^{2m}_s \pd_t \gm(0,t) = \pd^{2m}_s \pd_t \gm(L,t) =0
\end{align}
for any non-negative integer $m$. 
\end{lem}
\begin{proof}
By virtue of Lemma \ref{commu-2}, we have already known that 
\begin{align} \label{gm-t-1}
\pd_t \pd^n_s \gm = \pd^n_s \pd_t \gm - \sum^{n-1}_{j=0}\pd^j_s(\Gl \pd^{n-j}_s \gm), 
\end{align}
where
\begin{align} \label{gm-t-2}
\Gl &= \pd^2_s \left( \av{\pd^2_s \gm}^2 \right) - 2 \av{\pd^3_s \gm}^2 
          + \left(3 \av{\pd^2_s \gm}^2 - \lm^2 \right) \av{\pd^2_s \gm}^2 \\
 &= \pd^2_s \gm \cdot \pd^4_s \gm + \left(3 \av{\pd^2_s \gm}^2 - \lm^2 \right) \av{\pd^2_s \gm}^2. \notag
\end{align}
Moreover Lemma \ref{l-odd-k} gives us 
\begin{align} \label{gm-t-3} 
\pd_t \pd^{2m}_s \gm(0,t)= \pd_t \pd^{2m}_s \gm(L,t)=0. 
\end{align}
Since 
\begin{align*}
\pd^j_s(G^\lm \pd^{n-j}_s \gm)
 = \sum^{j}_{k=0} \mbox{}_j C_k \pd^k_s \Gl \pd^{2m-k}_s \gm,  
\end{align*}
Lemma \ref{commu-2} and \eqref{gm-t-2} yield that 
\begin{align} \label{gm-t-4}
\pd^j_s(\Gl \pd^{n-j}_s \gm) = 0
\end{align}
at $x=0$, $L$ for any $t>0$ and non-negative integer $j \leq n$. 
By \eqref{gm-t-1}, \eqref{gm-t-3}, and \eqref{gm-t-4}, we complete the proof. 
\end{proof}
By virtue of Lemma \ref{gm-t-bdc}, we are able to apply Lemma \ref{l-inter-p} to $\pd_t \gm$. 

In the rest of this section, we shall use the notation 
\begin{align*}
\int_\gm {\bf u} \cdot {\bf v} \, ds = \Ip{\bf{u}}{\bf{v}}, 
\end{align*}
where ${\bf u}$ and ${\bf v}$ are functions defined on $\gm$. 
By way of Lemma \ref{gm-t-bdc}, we obtain the following: 
\begin{lem} \label{l-conv-0}
For any $n \in \N$, it holds that 
\begin{align}
\Lns{\pd^n_s \pd_t \gm}{2} \to 0 \quad \text{as} \quad t \to \infty. 
\end{align}
\end{lem}
\begin{proof}
To begin with, we have 
\begin{align*}
\int^\infty_0 \Lns{\pd_t \gm}{2}^2 \, dt 
 = - \int^\infty_0 \pd_t \left( \int_\gm \vk^2 \, ds + \lm^2 \mL(\gm) \right) \, dt
 = \left[ \int_\gm \vk^2 \, ds + \lm^2 \mL(\gm) \right]^{t=0}_{t=\infty} < + \infty. 
\end{align*}
Next it follows from \eqref{pt-ds-G} that  
\begin{align} \label{gm-dt-1}
\pd_t \Lns{\pd_t \gm}{2}^2 
 = 2 \Ip{\pd_t \gm}{\pd_t (\pd_t \gm)} + \Ip{\pd_t \gm}{\Gl \pd_t \gm}. 
\end{align}
Making use of \eqref{gf}, Lemma \ref{commu-2}, and the relation $\pd_t \gm \cdot \pd_s \gm =0$, 
we obtain 
\begin{align*}
\pd_t \gm \cdot \pd_t(\pd_t \gm) 
 = & -2 \pd_t \cdot \pd^4_s \gm + 2 \pd_t \gm \cdot \sum^{3}_{j=0}\pd^j_s(\Gl \pd^{4-j}_s \gm) 
    + (\lm^2 - 3 \av{\pd^2_s \gm}^2) \pd_t \gm \cdot (\pd^2_s \pd_t \gm - 2 \Gl \pd^2_s \gm )  \\
    & \qquad -3 \pd_t (\av{\pd^2_s \gm }^2) \pd_t \gm \cdot \pd^2_s \gm 
       -3 \pd_s (\av{\pd^2_s \gm }^2) \pd_t \gm \cdot \pd_s \pd_t \gm.  
\end{align*}
By integrating by parts, \eqref{gm-dt-1} is reduced to 
\begin{align} \label{gm-dt-2}
\pd_t \Lns{\pd_t \gm}{2}^2 &= 
 -4 \Lns{\pd^2_s \pd_t \gm}{2}^2 + 4 \sum^{3}_{j=0} \Ip{\pd_t \gm}{\pd^j_s(\Gl \pd^{4-j}_s \gm)} \\ 
& \qquad - 2 \Ip{\lm^2 - 3 \av{\pd^2_s \gm}^2}{\av{\pd_s \pd_t \gm }^2 - 2 | \Gl |^2}  
         + \Ip{ 6 \pd_t (\av{\pd^2_s \gm}^2 ) + \av{\pd_t \gm}^2}{\Gl} \notag \\
& := -4 \Lns{\pd^2_s \pd_t \gm}{2}^2 + 4 I_1 - 2 I_2 + I_3. \notag        
\end{align}
We shall estimate the right-hand side. 
First, by Lemmas \ref{l-gm-bd-m}, \ref{l-est-deri-4}, and \ref{gm-t-bdc}, we have 
\begin{align} \label{gm-dt-2-1}
\av{I_2}  
 & \leq C ( 1 + \Lns{\pd^2_s \gm}{\infty}^2 ) \Lns{\pd_s \pd_t \gm}{2}^2 
   \leq C (1 + \Lns{\pd^4_s \gm}{2}^{\frac{1}{2}} ) \Lns{\pd_t \gm}{2} \Lns{\pd^2_s \pd_t \gm}{2} \\
 & \leq \ve \Lns{\pd^2_s \pd_t \gm}{2}^2 + C \Lns{\pd_t \gm}{2}^2 (1 + \Lns{\pd_t \gm}{2}). \notag
\end{align}
Next we turn to the estimate of $I_3$. Since 
\begin{align*}
\pd_t (\av{\pd^2_s \gm}^2) 
 = 2 \pd^2_s \gm \cdot \pd^2_s \pd_t \gm - 4 \av{\pd^2_s \gm}^2 \Gl, 
\end{align*}
we obtain 
\begin{align} \label{gm-dt-2-2}
\av{I_3} 
 & \leq 12 \Lns{\pd^2_s \gm}{\infty} \Lns{\pd^2_s \pd_t \gm}{2} \Lns{\pd_s \pd_t \gm}{2}  
       + 24 \Lns{\pd^2_s \gm}{\infty}^2 \Lns{\pd_s \pd_t \gm}{2}^2 \\ 
 & \qquad + \Lns{\pd_t \gm}{\infty} \Lns{\pd_t \gm}{2} \Lns{\pd_s \pd_t \gm}{2} \notag \\
 & \leq C \Lns{\pd^4_s \gm}{2}^\frac{1}{4} \Lns{\pd_t \gm}{2}^\frac{1}{2} \Lns{\pd^2_s \pd_t \gm}{2}^\frac{3}{2}
              + C \Lns{\pd^4_s \gm}{2}^\frac{1}{2} \Lns{\pd_t \gm}{2} \Lns{\pd^2_s \pd_t \gm}{2} \notag \\
 & \qquad \quad + \sqrt{2} \Lns{\pd_t \gm}{2}^\frac{9}{4} \Lns{\pd^2_s \pd_t \gm}{2}^\frac{3}{4} \notag \\
 & \leq \ve \Lns{\pd^2_s \pd_t \gm}{2}^2 
               + C \Lns{\pd_t \gm}{2}^2 (1 + \Lns{\pd_t \gm}{2} + \Lns{\pd_t \gm}{2}^\frac{8}{5}). \notag
\end{align}
Finally we estimate the term $I_1$. By integrating by parts, $I_1$ is written as follows:    
\begin{align} \label{w-1}
I_1 = \sum^{2}_{j=0} \Ip{(-1)^j \pd^j_s \pd_t \gm}{\Gl \pd^{4-j}_s \gm}  
        + \Ip{\pd_t \gm}{\pd^3_s (\Gl \pd_s \gm)}.
\end{align}
For $j = 0, 1, 2$, we have 
\begin{align*}
\av{\Ip{\pd^j_s \pd_t \gm}{\Gl \pd^{4-j}_s \gm} } 
 & \leq \Lns{\pd^j_s \pd_t \gm}{2} \Lns{\pd_s \pd_t \gm}{2} \Lns{\pd^{4-j}_s \gm}{\infty} 
   \leq C \Lns{\pd^2_s \pd_t \gm}{2}^{\frac{j+1}{2}} \Lns{\pd_t \gm}{2}^{\frac{3-j}{2}} 
         \Lns{\pd^4_s \gm}{2}^{\frac{5-2j}{4}} \\
 & \leq \ve \Lns{\pd^2_s \pd_t \gm}{2}^2 + C \Lns{\pd_t \gm}{2}^2 \Lns{\pd^4_s \gm}{2}^{\frac{5-2j}{3-j}}. 
\end{align*}
Hence the first term in the right-hand side of \eqref{w-1} is estimated as follows:  
\begin{align*}
\av{\sum^{2}_{j=0} \Ip{(-1)^j \pd^j_s \pd_t \gm}{\Gl \pd^{4-j}_s \gm} } 
  \leq \ve \Lns{\pd^2_s \pd_t \gm}{2}^2 + C \Lns{\pd_t \gm}{2}^2 (1+ \Lns{\pd_t \gm}{2}^2) . 
\end{align*}
Furthermore the equality $\pd_t\gm \cdot \pd_s \gm =0$ yields that 
\begin{align*}
\Ip{\pd_t \gm}{\pd^3_s (\Gl \pd_s \gm)} 
 =  -3 \Ip{\pd_s \pd_t \gm}{\pd_s \Gl \pd^2_s \gm}   
      + \Ip{\pd_t \gm}{\Gl \pd^4_s \gm} . 
\end{align*}
Then we obtain 
\begin{align*}
\av{\Ip{\pd_s \pd_t \gm}{\pd_s \Gl \pd^2_s \gm} } 
 & \leq \av{\Ip{\pd_s \pd_t \gm \cdot \pd^2_s \gm} 
             {\pd^2_s \pd_t \gm \cdot \pd_s \gm + \pd_s \pd_t \gm \cdot \pd^2_s \gm} } \\
 & \leq \Lns{\pd^2_s \gm}{\infty} \Lns{\pd_s \pd_t \gm}{2} \Lns{\pd^2_s \pd_t \gm}{2} 
         + \Lns{\pd^2_s \gm}{\infty}^2 \Lns{\pd_s \pd_t \gm}{2}^2 \\
 & \leq \ve \Lns{\pd^2_s \pd_t \gm}{2}^2 + C \Lns{\pd_t \gm}{2}^2 (1+\Lns{\pd_t \gm}{2}),\\  
\av{\Ip{\pd_t \gm}{\Gl \pd^4_s \gm} }
 & \leq \Lns{\pd_t \gm}{\infty} \Lns{\pd_s \pd_t \gm}{2} \Lns{\pd^4_s \gm}{2} \\
 & \leq \sqrt{2} \Lns{\pd^2_s \pd_t \gm}{2}^\frac{3}{4} \Lns{\pd_t \gm}{2}^\frac{5}{4} \Lns{\pd^4_s \gm}{2} 
   \leq \ve \Lns{\pd^2_s \pd_t \gm}{2}^2 + C \Lns{\pd_t \gm}{2}^2 \Lns{\pd^4_s \gm}{2}^\frac{8}{5}.
\end{align*}
Hence we see that  
\begin{align} \label{gm-dt-2-3}
\av{I_1} 
 \leq \ve \Lns{\pd^2_s \pd_t \gm}{2}^2 + C \Lns{\pd_t \gm}{2}^2 (1+ \Lns{\pd_t \gm}{2}^2 ).
\end{align}
Letting $\ve>0$ sufficiently small and using \eqref{gm-dt-2}, \eqref{gm-dt-2-1}, \eqref{gm-dt-2-2}, and 
\eqref{gm-dt-2-3}, we have the inequality 
\begin{align} \label{d-ineq-1}
\pd_t \Lns{\pd_t \gm}{2}^2  
 \leq - \Lns{\pd^2_s \pd_t \gm}{2}^2 
        + C \Lns{\pd_t \gm}{2}^{2}(1 + \Lns{\pd_t \gm}{2}^2) . 
\end{align}
This implies that $\Lns{\pd_t \gm}{2} \to 0$ as $t \to +\infty$. 
In particular, $\Lns{\pd_t \gm}{2}$ is bonded for any $t > 0$. 
Then \eqref{d-ineq-1} is reduced to 
\begin{align} \label{d-ineq-2}
\pd_t \Lns{\pd_t \gm}{2}^2  
\leq - \Lns{\pd^2_s \pd_t \gm}{2}^2 + C \Lns{\pd_t \gm}{2}^{2}.  
\end{align}
Integrating \eqref{d-ineq-2} on $[0, \infty)$, we obtain 
\begin{align}
\int^{\infty}_0 \Lns{\pd^2_s \pd_t \gm}{2}^2 \, dt 
 \leq -\int^{\infty}_0 \pd_t \Lns{\pd_t \gm}{2}^2 \, dt 
    + C \int^{\infty}_0 \Lns{\pd_t \gm}{2}^{2} \, dt < \infty.     
\end{align}

Next, suppose that 
\begin{align*}
\int^\infty_0 \Lns{\pd^j_s \pd_t \gm}{2}^2\, dt < \infty, \qquad 
\Lns{\pd^{j-2}_s \pd_t \gm}{2} \to 0 \quad \text{as} \quad t \to \infty  
\end{align*}
hold for $2 \leq j \leq 2m$, where $m \geq 1$. 
From the assumption, 
we see that $\Lns{\pd^n_s \gm}{2}$ is bounded for any $t>0$ and $2 \leq n \leq 2m+2$. 
Since 
\begin{align*}
\pd_t \Lns{\pd^{2m}_s \pd_t \gm}{2}^2 
 & = \Ip{ 2 \pd^{2m}_s \pd_t \gm}{\pd_t \pd^{2m}_s \pd_t \gm} 
       + \Ip{\pd^{2m}_s \pd_t \gm}{\Gl \pd^{2 m}_2 \pd_t \gm} \\
 & = 2 \Ip{\pd^{4 m}_s \pd_t \gm}{\pd_t \pd_t \gm} 
        -2 \Ip{\pd^{2 m}_s \pd_t \gm}{\sum^{2 m-1}_{j=0} \pd^j_s(\Gl \pd^{2m-j}_s \pd_t \gm)}  
          + \Ip{\pd^{2m}_s \pd_t \gm}{ \Gl \pd^{2m}_s \pd_t \gm} \\
& := 2 I_1 + 2 I_2 + I_3,  
\end{align*}
it is sufficient to estimate the terms $I_1$, $I_2$, and $I_3$. 
Since $\Gl= \pd_s \pd_t \gm \cdot \pd_s \gm$, it is clear that 
\begin{align} \label{est-i3}
\av{I_3} 
 \leq C \Lns{\pd^{2m}_s \pd_t \gm}{2}^2. 
\end{align}
Concerning $I_2$, for $k= 0, 1, \cdots, 2m-1$, we have 
\begin{align*}
\av{\Ip{\pd^{2 m}_s \pd_t \gm}{ \pd^k_s \Gl \pd^{2m-k}_s \pd_t \gm} } 
 \leq C \Lns{\pd^{2m}_s \pd_t \gm}{2} \sum^{2m-1}_{l=m} \Lns{\pd^l_s \pd_t \gm}{2} 
  \leq \Lns{\pd^{2m}_s \pd_t \gm}{2}^2 + C \sum^{2m-2}_{l=m} \Lns{\pd^l_s \pd_t \gm}{2}^2. 
\end{align*}
Hence we obtain 
\begin{align} \label{est-i2}
\av{I_2} \leq \Lns{\pd^{2m}_s \pd_t \gm}{2}^2 + C \sum^{2m-2}_{l=m} \Lns{\pd^l_s \pd_t \gm}{2}^2. 
\end{align}
Concerning the term $I_1$, using \eqref{gf} and integrating by parts, $I_1$ is reduced to  
\begin{align*}
I_1 &=-2 \Lns{\pd^{2m+2}_s \pd_t \gm}{2}^2 
         + 2 \Ip{\pd^{4m}_s \pd_t \gm}{\sum^3_{j=0} \pd^j_s(\Gl \pd^{4-j}_s \pd_t \gm)} 
         + 3 \Ip{\pd^{4m+1}_s \pd_t \gm}{\pd_t (\av{\pd^2_s \gm}^2) \pd_s \gm} \\
      & \quad  - \Ip{\pd^{4m+1}_s \pd_t \gm}{(\lm^2 - 3 \av{\pd^2_s \gm}^2) \pd_s \pd_t \gm} 
         - \Ip{\pd^{4m}_s \pd_t \gm}{(\lm^2 - 3 \av{\pd^2_s \gm}^2) \sum^1_{j=0} \pd^j_s (\Gl \pd^{2-j}_s \gm)} \\
       & \qquad + 6 \Ip{\pd^{4m}_s \pd_t \gm}{\pd_s (\av{\pd^2_s \gm}^2) \Gl \pd_s \gm} \\
       &:= -2 \Lns{\pd^{2m+2}_s \pd_t \gm}{2}^2 + 2 I_{11} + 3 I_{12} - I_{13} - I_{14} + 6 I_{15}. 
\end{align*}
First we estimate $I_{12}$. 
Since 
\begin{align*}
\pd_t (\av{\pd^2_s \gm}^2 ) = 2 \pd^2_s \pd_t \gm \cdot \pd^2_s \gm + 4 \Gl \av{\pd^2_s \gm}^2, 
\end{align*}
we have 
\begin{align*} 
\av{I_{12}} & \leq 2 \av{\Ip{\pd^{4m+1}_s \pd_t \gm}{\pd^2_s \pd_t \gm \cdot \pd^2_s \gm}} 
                    + 4 \av{\Ip{\pd^{4m+1}_s \pd_t \gm}{\Gl \av{\pd^2_s \gm}^2}} \\
            & = 2 \av{\Ip{\pd^{2m+2}_s \pd_t \gm}{\pd^{2m-1}_s (\pd^2_s \pd_t \gm \cdot \pd^2_s \gm)}} 
                    + 4 \av{\Ip{\pd^{2m+2}_s \pd_t \gm}{\pd^{2m-1}_s (\Gl \av{\pd^2_s \gm}^2)}} \\
            & \leq C \Lns{\pd^{2m+2}_s \pd_t \gm}{2} \sum^{2m+1}_{j=1} \Lns{\pd^j_s \pd_t \gm}{2} \\
            & \leq \ve \Lns{\pd^{2m+2}_s \pd_t \gm}{2}^2 + C \Lns{\pd^{2m}_s \pd_t \gm}{2}^2 
                    + C \sum^{2m-2}_{j=1} \Lns{\pd^j_s \pd_t \gm}{2}^2. 
\end{align*}
Along the same line, we obtain 
\begin{align*} 
\max\{\av{I_{13}}, \av{I_{14}}, \av{I_{15}} \} 
 \leq \ve \Lns{\pd^{2m+2}_s \pd_t \gm}{2}^2 + C \Lns{\pd^{2m}_s \pd_t \gm}{2}^2 
        + C \sum^{2m-2}_{j=1} \Lns{\pd^j_s \pd_t \gm}{2}^2.
\end{align*}
Finally we turn to the estimate of $I_{11}$. 
We reduce the term $I_{11}$ to 
\begin{align*}
I_{11} & = 2 \sum^{2}_{j=0} \Ip{\pd^{4m}_s \pd_t \gm}{\pd^j_s (\Gl \pd^{4-j}_s \gm)}
            + 2 \Ip{\pd^{4m}_s \pd_t \gm}{\pd^3_s (\Gl \pd_s \gm)}
         := 2 J_1 + 2 J_2. 
\end{align*}
Moreover, by virtue of the relation $\pd_s \pd_t \gm \cdot \pd_s \gm =0$, $J_2$ is reduced to 
\begin{align*}
J_2 &= \Ip{\pd_t \gm}{\pd^{4m+3}_s (\Gl \pd_s \gm)} 
     = \Ip{\pd_t \gm}{\pd^{4m+3}_s (\Gl \pd_s \gm) - \pd^{4m+3}_s \Gl \pd_s \gm} \\
    &= \Ip{\pd^{2m+2}_s \pd_t \gm}{\pd^{2m+1}_s (\Gl \pd_s \gm)}
        - \Ip{\pd_t \gm}{\pd^{4m+3}_s \Gl \pd_s \gm} 
     := J_{21} - J_{22}.  
\end{align*} 
Since 
\begin{align*}
J_{21} &= \Ip{\pd^{2m+2}_s \pd_t \gm}{\pd^{2m+1}_s \Gl \pd_s \gm} 
           + \sum^{2m}_{j=0}\Ip{\pd^{2m+2}_s \pd_t \gm}{\pd^j_s \Gl \pd^{2m+1-j}_s \gm}, 
\end{align*}
and 
\begin{align*}
J_{22} &= \Ip{\pd^{2m+2}_s (\pd_t \gm \cdot \pd_s \gm)}{\pd^{2m+1}_s \Gl} \\
       &= \Ip{\pd^{2m+2}_s \pd_t \gm}{\pd^{2m+1}_s (\Gl \pd_s \gm)} 
           + \sum^{2m+1}_{j=0} \Ip{\pd^j_s \pd_t \gm}{\pd^{2m+1}_s \Gl \pd^{2m+3-j}_s \gm}, 
\end{align*}
$J_2$ is written as follows: 
\begin{align*}
J_2 = \sum^{2m}_{j=0}\Ip{\pd^{2m+2}_s \pd_t \gm}{\pd^j_s \Gl \pd^{2m+1-j}_s \gm} 
       - \sum^{2m+1}_{j=0} \Ip{\pd^j_s \pd_t \gm}{\pd^{2m+1}_s \Gl \pd^{2m+3-j}_s \gm}
    := K_1 + K_2. 
\end{align*}
Here we have 
\begin{align*}
\av{K_1} & \leq C \Lns{\pd^{2m+2}_s \pd_t \gm}{2} \sum^{2m+1}_{j=0} \Lns{\pd^j_s \pd_t \gm}{2} \\
         & \leq \ve \Lns{\pd^{2m+2}_s \pd_t \gm}{2}^2 + C \Lns{\pd^{2m}_s \pd_t \gm}{2}^2 
                 + C \sum^{2m-2}_{j=0} \Lns{\pd^j_s \pd_t \gm}{2}^2,  
\end{align*}
and 
\begin{align*}
\av{K_2} & \leq \sum^{2m+1}_{j=0} 
                \av{\Ip{\pd^j_s \pd_t \gm \cdot \pd^{2m+3-j}_s \gm}{\pd^{2m+1}_s(\pd_s \pd_t \gm \cdot \pd_s \gm)}}\\
         & \leq C \sum^{2m+1}_{j=1} \Lns{\pd^j_s \pd_t \gm}{2} \cdot 
                  \sum^{2m+2}_{k=1} \Lns{\pd^k_s \pd_t \gm}{2} \\
         & \leq \ve \Lns{\pd^{2m+2}_s \pd_t \gm}{2}^2 + C \Lns{\pd^{2m}_s \pd_t \gm}{2}^2 
                 + C \sum^{2m-2}_{l=0} \Lns{\pd^l_s \pd_t \gm}{2}^2. 
\end{align*}
Hence we obtain 
\begin{align} \label{est-j1}
\av{J_2} \leq \ve \Lns{\pd^{2m+2}_s \pd_t \gm}{2}^2 + C \Lns{\pd^{2m}_s \pd_t \gm}{2}^2 
                 + C \sum^{2m-2}_{l=0} \Lns{\pd^l_s \pd_t \gm}{2}^2. 
\end{align}
Along the same line, we get 
\begin{align} \label{est-j2}
\av{J_1} & \leq \ve \Lns{\pd^{2m+2}_s \pd_t \gm}{2}^2 + C \Lns{\pd^{2m}_s \pd_t \gm}{2}^2 
                 + C \sum^{2m-2}_{l=0} \Lns{\pd^l_s \pd_t \gm}{2}^2. 
\end{align}
The estimates \eqref{est-j1} and \eqref{est-j2} imply 
\begin{align}
\av{I_{11}} \leq \ve \Lns{\pd^{2m+2}_s \pd_t \gm}{2}^2 + C \Lns{\pd^{2m}_s \pd_t \gm}{2}^2 
                 + C \sum^{2m-2}_{l=0} \Lns{\pd^l_s \pd_t \gm}{2}^2. 
\end{align}
Therefore, letting $\ve>0$ sufficiently small, we see that 
\begin{align} \label{est-key-1}
\pd_t \Lns{\pd^{2m}_s \pd_t \gm}{2}^2 
 \leq - \Lns{\pd^{2m+2}_s \pd_t \gm}{2}^2 + C \Lns{\pd^{2m}_s \pd_t \gm}{2}^2 
                 + C \sum^{2m-2}_{l=0} \Lns{\pd^l_s \pd_t \gm}{2}^2. 
\end{align}
Integrating \eqref{est-key-1} with respect to $t$ on $[0,\infty)$, we have 
\begin{align*}
\int^\infty_0 \Lns{\pd^{2m+2}_s \pd_t \gm}{2}^2 \, dt < \infty. 
\end{align*} 
This completes the proof. 
\end{proof}
Making use of Lemmas \ref{l-est-deri-4}, \ref{l-est-deri-n}, and \ref{l-conv-0}, 
we prove the following: 
\begin{thm} \label{c-comv-thm}
Let $\gm$ be a solution of \eqref{comp-sys}. 
Then there exist a sequence $\{ t_i \}^{\infty}_{i=0}$ with $t_i \to \infty$ 
and a planar curve $\hat{\gm}$ such that $\gm(\cdot, t_i)$ converges to $\hat{\gm}(\cdot)$ 
up to a reparametrization in the $C^\infty$ topology as $t_i \to \infty$ . 
Moreover $\hat{\gm}$ is a stationary solution of \eqref{comp-sys}.  
\end{thm}
\begin{proof}
Since it holds that 
\begin{align*}
R \leq \mL(\gm(\cdot,t)) 
  \leq \dfrac{1}{\lm^2} \left\{ \int^L_0 k_0^2 \, dx - \int_\gm \vk^2 \, ds  \right\} + \mL(\gm_0) < C, 
\end{align*}
we reparameterize $\gm$ by its arc length, i.e., $\gm= \gm(s,t)$. 
By virtue of Lemmas \ref{l-est-deri-4}, \ref{l-est-deri-n}, and \ref{l-conv-0}, we see that 
\begin{align} \label{c-1}
\Lns{\pd^n_s \gm(\cdot,t)}{2} < \infty 
\end{align}
for any integers $n \geq 2$. From Lemma \ref{l-inter-p}, the inequality \eqref{c-1} yields 
\begin{align*}
\Lns{\pd^n_s \gm(\cdot,t)}{\infty} < \infty.  
\end{align*}
Thus $\pd^n_s \vk$ is uniformly bounded with respect to $t$ for any non-negative integers $n$. 
Furthermore it follows from \eqref{c-1} that 
\begin{align*}
\av{\pd^n_s \vk(s_1,t) - \pd^n_s \vk(s_2,t)} 
 \leq \av{\int^{s_1}_{s_2} \pd^{n+1}_s \vk(s,t) \, ds} 
 \leq C \av{s_1 - s_2},  
\end{align*}
for each $n \in \N$, where the constant $C$ is independent of $t$. 
Thus $\pd^n_s \vk$ is equi-continuous with respect to $t$. 
Thus, there exist a sequence $\{ t_{1,j} \}_{j=1}^\infty$ and $\hat{\vk}(x)$ such that 
$\vk(\cdot,t_{1,j})$ uniformly converges to $\hat{\vk}(\cdot)$ as $t_{1,j} \to \infty$. 
Similarly, for each $n \in \N$, there exists a subsequence 
$\{t_{n,j} \}_{j=1}^\infty \subset \{ t_{n-1,j} \}_{j=1}^\infty$ such that $\pd^n_s \vk(\cdot,t)$ 
uniformly converges to $\pd^n_{\cdot} \hat{\vk}(\cdot)$ as $t_{n,j} \to \infty$. 
By virtue of the diagonal method, we see that there exist a sequence $\{ t_i \}_{i=1}^\infty$ and 
a function $\hat{\vk}(\cdot)$ such that $\vk(\cdot,t_i)$ converges to $\hat{\vk}(\cdot)$ in the 
$C^\infty$ topology. 
Since $\gm(\cdot,t)$ is fixed at the boundary, a curve $\hat{\gm}$ with curvature $\hat{\vk}$ is 
uniquely determined. Moreover, by Lemma \ref{l-conv-0}, $\pd_t \gm(\cdot,t)$ uniformly converges 
to $0$ as $t \to \infty$. Therefore the curve $\hat{\gm}$ is a stationary solution of \eqref{comp-sys}. 
\end{proof}

\section{Non compact case} \label{n-comp-case}
Let $\gm_0(x)= (\phi_0(x), \psi_0(x)) : \R \to \R^2$ be a smooth curve, 
and $\vk_0$ denote the curvature. 
Let $\gm_0(x)$ satisfy the following conditions: 
\begin{gather} 
\av{{\gm_0}'(x)} \equiv 1 \tag{A1} \label{cond-1} \\
\pd^m_x \vk_0 \in L^2(\R) \quad \text{\rm for all} \,\,\,\,  m \geq 0, \tag{A2} \label{cond-2} \\ 
\lmt{x \to \infty}{\phi_0(x)}= \infty, \quad \lmt{x \to -\infty}{\phi_0(x)}= -\infty,
\quad \lmt{\av{x} \to \infty}{\phi'_0(x)}=1, 
\tag{A3} \label{cond-3} \\ 
\psi_0(x) = O(x^{-\va}) \,\,\,\text{for \,some}\,\,\,\va > \frac{1}{2} \,\,\, \text{as} \,\,\, 
\av{x} \to \infty,\quad  \psi'_0 \in L^2(\R). \tag{A4} \label{cond-4}
\end{gather}
The definition of $\gm_0$ and \eqref{cond-1} imply that $\gm_0$ has infinite length. 
From \eqref{cond-2}, we see that $\gm_0$ approaches a straight line as $\av{x} \to \infty$. 
Furthermore \eqref{cond-3} and \eqref{cond-4} yield that the straight line is given by the axis. 
Indeed, by \eqref{cond-2} and \eqref{cond-3}, for sufficiently small $\rho>0$, 
there exists a constant $M>0$ such that  
\begin{align} \label{def-M}
\su{\av{x} \in (M,\infty)}{\av{\av{\phi'_0(x)} - 1}} < \rho, \quad \sup_{\av{x} \in (M, \infty)}\av{\psi_0(x)} < \rho, \quad 
\su{\av{x} \in (M,\infty)}{\av{\psi'_0(x)}} < \rho.  
\end{align}

To begin with, we prove that the shortening-straightening flow starting from $\gm_0$ has a 
classical solution for any finite time. 
As the first step, we shall construct an ``approximate solution". 
For this purpose, it starts from the definition of a cut-off function 
$\eta_r(x) \in C^\infty_c(\R)$: 
\begin{align*}
\eta_r(x) = 1 \quad & \text{for \, any} \quad \av{x} \in [0, r-1], \\
0 < \eta_r(x) < 1 \quad & \text{for \, any} \quad \av{x} \in (r-1, r), \\
\eta_r(x) = 0 \quad & \text{for \, any} \quad \av{x} \in [r, +\infty). 
\end{align*}
Using the cut-off function, we define a curve $\Gm_{0,r} : [-r, r] \to \R^2$ as  
\begin{align*}
\Gm_{0,r}(x)= (\phi_0(x), \eta_r(x) \psi_0(x) ) \Bigm|_{x \in [-r,r]}, 
\end{align*}
and we consider the following initial-boundary value problem: 
\begin{align} \label{comp-sys-r}
\begin{cases}
& \pd_t \gm  = (\lm^2 \vk - 2 \pd^2_s \vk - \vk^3 ) \bn, \\
& \gm(-r,t) = (\phi_0(-r),0), \quad \gm(r,t)  = (\phi(r), 0), \quad 
  \vk(-r,t)= \vk(r,t) = 0, \\ 
& \gm(x,0)=\Gm_{0,r}(x). 
\end{cases} \tag{${\rm SS}_r$}
\end{align}
We are able to verify that the compatibility condition of \eqref{comp-sys-r} holds. 
\begin{lem} \label{conv-as-a-infinity}
Let $r > M$. Then $\Gm_{0,r}(x)$ is smooth and satisfies 
\begin{align} \label{bd-cond-r}  
\Gm_{0,r}(-r)=(\phi_0(-r),0), \quad 
\Gm_{0,r}(r,0)= (\phi_0(r),0), \quad 
\vk_{0,r}(-r)= \vk_{0,r}(r)= 0, 
\end{align}
where $\vk_{0,r}$ denotes the curvature of $\Gm_{0,r}$. 
\end{lem}
\begin{proof}
Let $r > M$. 
By the definition of $\eta_r$, it is clear that $\Gm_{0,r}$ is smooth and 
$\Gm_{0,r}(-r)=(\phi_0(-r),0)$, 
$\Gm_{0,r}(r,0)= (\phi_0(r),0)$ hold. 
Furthermore, since the curvature $\vk_{0,r}(x)$ is written as 
\begin{align*}
\dfrac{\mathfrak{R}(\phi'_0(x), \pd_x \eta_r(x) \psi_0(x)+ \eta_r(x) \psi'_0(x))
          \cdot(\phi''_0(x), \pd^2_x \eta_r(x) \psi_0(x) + 2 \eta'_r(x) \psi'_0(x)  
                   + \eta_r(x) \psi''_0(x))}{\av{\Gm'_{0,r}(x)}^3}, 
\end{align*}
we observe that $\vk_{0,r}(-r)$ and $\vk_{0,r}(r)$ vanish. 
\end{proof}

Concerning \eqref{comp-sys-r}, we obtain the following: 
\begin{lem} \label{exist-pre-sol}
Let $r>M$. Then there exists a unique classical solution 
of \eqref{comp-sys-r} for any time $t>0$. 
Moreover, there exists a sequence $\{ t_i \}^\infty_{i=0}$ with $t_i \to \infty$ such that 
the solution converges to a stationary solution of \eqref{comp-sys-r} as $t_i \to \infty$ 
up to a reparametrization. 
\end{lem}
\begin{proof}
Lemma \ref{conv-as-a-infinity} and Theorem \ref{comp-gm-global} gives us the conclusion. 
\end{proof} 

In what follows, let $\gm_r(x,t)$ denote the solution of \eqref{comp-sys-r}, 
and $\vk_r(x,t)$ be the curvature of $\gm_r(x,t)$. 
In order to construct a solution of \eqref{s-s}, 
we apply Arzel\`a-Ascoli's theorem to $\{ \gm_r \}_{r > M}$. 
The point is to prove that $\vk_r$ is uniformly bounded with respect to $r$.  
\begin{lem} \label{bd-vk-r} 
There exists a positive constant $C$ being independent of $r$ such that 
\begin{align} 
\su{r \in (M,\infty)}{\Lns{\vk_r(t)}{2}} < C 
\end{align}
 for any $t>0$. 
\end{lem}
\begin{proof}
Let $r > M$. 
First recall that the inequality 
\begin{align} \label{up-bd}
\Lns{\vk_r}{2}^2 \leq 
\Lns{\vk_{0,r}}{2}^2 + \lm^2 \left\{ \mL(\Gm_{0,r}) - (\phi_0(r) - \phi_0(-r)) \right\}
\end{align}
holds. Concerning the first term of the right-hand side of \eqref{up-bd}, it holds that 
\begin{align*}
\Lns{\vk_{0,r}}{2}^2 
 &= \int^r_{-r} \av{\vk_{0,r}(x)}^2 \av{\pd_x \Gm_{0,r}(x)} \, dx \\
 &= \int^{-r+1}_{-r} \av{\vk_{0,r}(x)}^2 \av{\pd_x \Gm_{0,r}(x)} \, dx
     + \int^{r-1}_{-r+1} \av{\vk_0(x)}^2 \, dx 
     + \int^r_{r-1} \av{\vk_{0,r}(x)}^2 \av{\pd_x \Gm_{0,r}(x)} \, dx \\
 & \leq \Ln{\vk_0}{\R}^2 
    + \int^{-r+1}_{-r} \av{\vk_{0,r}(x)}^2 \av{\pd_x \Gm_{0,r}(x)} \, dx
    + \int^r_{r-1} \av{\vk_{0,r}(x)}^2 \av{\pd_x \Gm_{0,r}(x)} \, dx. 
\end{align*}
By virtue of Frenet-Serret's formula, \eqref{cond-1}, and \eqref{cond-2}, 
we see that $\phi^{(m)}$, $\psi^{(m)} \in L^2(\R)$ for any integer $m \geq 2$. 
Combining the fact with the expression of $\vk_{0,r}$, we see that 
\begin{align*}
\int_{[-r, -r+1] \cup [r-1,r]} \av{\vk_{0,r}(x)}^2 \av{\pd_x \Gm_{0,r}(x)} \, dx <  C, 
\end{align*}
where the constant $C$ depends only on $\gm_0$. 
This yields that 
\begin{align*}
\Lns{\vk_{0,r}}{2}^2 < \Ln{\vk_0}{\R}^2 + C 
\end{align*}
holds for any $r > M$. 

In order to obtain the conclusion, we turn to a estimate of the second term 
in the right-hand side of \eqref{up-bd}. 
Let us fix $b \in (M, r-1)$ arbitrarily. Then we have 
\begin{align} \label{key-est}
& \mL(\Gm_{0,r}) - (\phi_0(r)-\phi_0(-r)) \\
& \quad = \left\{ \mL_1(\Gm_{0,r}) - (\phi_0(b)-\phi_0(-b)) \right\}  
      + \left\{ \mL_2^+(\Gm_{0,r}) - (\phi_0(r)-\phi_0(b)) \right\} \notag \\ 
& \qquad \quad + \left\{ \mL_2^- (\Gm_{0,r}) - (\phi_0(-b)-\phi_0(-r)) \right\},  \notag
\end{align}
where 
\begin{align*}
\mL_1(\Gm_{0,r})&= \int^b_{-b} \av{\pd_x \Gm_{0,r}(x)} \, dx = 2b, \\ 
\mL_2^+(\Gm_{0,r})&= \int^r_b \av{\pd_x \Gm_{0,r}(x)} \, dx, \\
\mL_2^-(\Gm_{0,r})&= \int^{-b}_{-r} \av{\pd_x \Gm_{0,r}(x)} \, dx. 
\end{align*}
The first term in the right hand side of \eqref{key-est} is bounded independently of $r$. 
In the following, we focus on the second term.  

From \eqref{def-M}, for any $r > M$, we see that 
$\Gm_{0,r}(x)$ is expressed as a variation of line in the interval $[b, r]$. 
Here we derive a variational formula for $\mL$ in a general case. 
Let $\Gm(x) : [b,r] \to \R^2$ be a straight line. 
For $\vp \in C^{\infty}((-\ve_0, \ve_0); C^\infty[b,r])$ 
with $\vp(x,0) \equiv 0$ and $\vp(r, \ve) \equiv 0$, 
we consider a variation 
\begin{align*}
\Gm(x,\ve) = \Gm(x) + \vp(x,\ve), 
\end{align*}
where $\Gm(b, \ve)$ is on the straight line orthogonally intersecting with $\Gm(x)$ 
at $x=b$ for any $\ve>0$. Concerning the variation, it holds that  
\begin{align} \label{expan-1}
\mL(\Gm(\cdot,\ve)) 
  &= \mL(\Gm(\cdot)) + \dfrac{d}{d \ve} \mL(\Gm(\cdot, \ve)) \bigm|_{\ve=0} \ve 
       + \dfrac{d^2}{d \ve^2} \mL(\Gm(\cdot,\ve)) \bigm|_{\ve=\vs} \ve^2, 
\end{align}
where $\av{\vs} < \av{\ve}$. Concerning the first variation, we have 
\begin{align} \label{1st-vari-1}
\dfrac{d}{d \ve} \mL(\Gm(\cdot, \ve)) 
= \int^r_b 
  \dfrac{\{ \Gm'(x) + \vp'(x,\ve) \} \cdot \vp'_\ve (x,\ve)}{\av{\Gm'(x) + \vp'(x,\ve)}}\, dx. 
\end{align}
Integrating by parts and letting $\ve=0$, \eqref{1st-vari-1} is reduced to 
\begin{align*}
\dfrac{d}{d \ve} \mL(\Gm(\cdot, \ve)) \biggm|_{\ve=0}
 = \left[ \dfrac{\pd_x \Gm(x)}{\av{\pd_x \Gm(x)}} \cdot \vp_\ve (x,0) \right]^{r}_{b}
   -\int^r_b \left( \dfrac{\Gm'(x)}{\av{\Gm'(x)}} \right)' \cdot \vp_\ve(x,0) \, dx = 0. 
\end{align*}
For $\Gm(x)$ is a straight line. 
Next, concerning the second variation, we have 
\begin{align*}
\dfrac{d^2}{d \ve^2} \mL(\Gm(\cdot, \ve)) 
 = \int^r_b \left\{ \dfrac{\av{\vp'_\ve(x,\ve)}^2}{\av{\Gm'(x,\ve)}} 
     - \dfrac{(\Gm'(x,\ve) \cdot \vp'_\ve(x,\ve))^2}{\av{\Gm'(x,\ve)}^3} \right\} \, dx . 
\end{align*}
Here, in particular, we set 
\begin{align} \label{setting}
\Gm(x) = (\phi_0(x),0), \quad \vp(x,\ve) 
       = \left( 0, \frac{2 \ve}{\ve_0} \eta_r(x) \psi_0(x) \right).  
\end{align} 
Since $\Gm(x,\ve_0/2) = \Gm_{0,r}(x)$, the relation \eqref{expan-1} gives us the following: 
\begin{align} \label{expan-2}
\mL_2^+(\Gm_{0,r}) - \left\{ \phi_0(r)-\phi_0(b) \right\} 
& \leq \dfrac{{\ve_0}^2}{2} \int^r_b \dfrac{\av{\vp'_\ve(x,\vs)}^2}{\av{\Gm'(x,\vs)}} \, dx. 
\end{align}
Under \eqref{setting}, we have $\av{\Gm'(x, \vs)} > 1- \rho$. 
Thus the right hand side of \eqref{expan-2} is estimated as follows:   
\begin{align*}
\int^r_b \dfrac{\av{\vp'_\ve(x,\vs)}^2}{\av{\Gm'(x,\vs)}} \, dx 
   \leq C \int^r_b \left\{ \av{\psi_0(x)}^2 + \av{\psi'_0(x)}^2 \right\} \, dx. 
\end{align*}
Consequently we see that  
\begin{align} \label{expan-4}
\mL^+_2(\Gm_{0,r}) - \left\{ \phi_0(r)-\phi_0(b) \right\}  
 \leq C \int^\infty_b \left\{ \av{\psi_0(x)}^2 + \av{\psi'_0(x)}^2 \right\} \, dx. 
\end{align}
Along the same line as above, we find 
\begin{align} \label{expan-5}
\mL^-_2(\Gm_{0,r}) - \left\{ \phi_0(-b)-\phi_0(-r) \right\}  
 \leq C \int^{-b}_{-\infty} \left\{ \av{\psi_0(x)}^2 + \av{\psi'_0(x)}^2 \right\} \, dx. 
\end{align}
Combining the estimates \eqref{expan-4}-\eqref{expan-5} with condition \eqref{cond-3}, we obtain 
\begin{align*}
\su{r \in (M, \infty)}{\left\{\mL(\Gm_{0,r}) - \left\{ \phi_0(r) - \phi(-r) \right\} \right\}} 
   < \infty.  
\end{align*}
This implies $\su{r \in (M, \infty)}{\Lns{\vk_r}{2}} < \infty$. 
\end{proof}
Making use of Lemma \ref{bd-vk-r}, we obtain a estimate for $\Lns{\pd^m_s \vk_r}{2}$: 
\begin{lem} \label{bd-m-vk-r}
Let $r > M$. Then, for any $m \in \N$, there exist constants $C_1>0$ and $C_2>0$ being 
independent of $r$ such that  
\begin{align*}
\su{r \in (M, \infty)}{\Lns{\pd^m_s \vk_r(t)}{2}} 
 \leq C_1 + C_2 t.   
\end{align*}
\end{lem}
\begin{proof}
Let $r > M$. Along the same line as in the proof of Lemma \ref{high-energy}, we have 
\begin{align*}
\dfrac{d}{dt} \Lns{\pd^m_s \vk_r}{2}^2  
 \leq C \Lns{\vk_r}{2}^{4m+6} + C \Lns{\vk_r}{2}^{4m+10} . 
\end{align*}
Then, Lemma \ref{bd-vk-r} yields that 
\begin{align*}
\Lns{\pd^m_s \vk_r(t)}{2}^2 
 \leq \Lns{\pd^m_s \vk_{0,r}}{2}^2 + C t.  
\end{align*}
Since $\Lns{\pd^m_s \vk_{0,r}}{2} \leq \Ln{\pd^m_x \vk_0}{\R} + C$, we obtain the conclusion. 
\end{proof}

Next we show estimates on the local length of $\gm_r$: 
\begin{lem} \label{ele-bd-lm}
Let $T>0$ be any positive number. 
Then there exist positive constants $C_1$ and $C_2$ being independent of $r$ 
such that the inequalities 
\begin{align} 
& \dfrac{1}{C_1(T,\gm_0)} \leq \av{\pd_x \gm_r(x,t)} \leq C_1(T,\gm_0), \label{ele-bd-1} \\
& \av{\pd^m_x \av{\pd_x \gm_r(x,t)}} \leq C_2(T, \gm_0) \label{ele-bd-2}
\end{align}
hold for any $(x,t) \in [-r, r] \times [0, T]$ and any integer $m > 1$.  
\end{lem}
\begin{proof} 
First we prove \eqref{ele-bd-1}. Since 
\begin{align*}
\pd_x \pd_t \gm_r 
& = \pd_x \left( -2 \pd^2_s \vk_r - {\vk_r}^3 + \lm^2 \vk_r \right) \bn_r 
     + \left( -2 \pd^2_s \vk_r - {\vk_r}^3 + \lm^2 \vk_r \right) \pd_x \bn_r, 
\end{align*}
and $\pd_x \bn_r = \av{\pd_x \gm_r} \pd_s \bn_r = - \av{\pd_x \gm_r} \vk_r \pd_s \gm_r$, 
we have 
\begin{align} \label{eq-le-1}
\pd_t \av{\pd_x \gm_r} 
 = \dfrac{\pd_x \gm_r \cdot \pd_x \pd_t \gm_r}{\av{\pd_x \gm_r}} 
 = - \vk_r \left( -2 \pd^2_s \vk_r - {\vk_r}^3 + \lm^2 \vk_r \right) \av{\pd_x \gm_r}. 
\end{align}
Thus $\av{\pd_x \gm_r}$ satisfies the initial value problem 
\begin{align} \label{od-eq-1}
\begin{cases}
\dfrac{d u}{dt} = F(\vk_r) u, \\
u(0) = 1, 
\end{cases}
\end{align}
where 
\begin{align*}
F(\vk_r) = - \vk_r \left( -2 \pd^2_s \vk_r - {\vk_r}^3 + \lm^2 \vk_r \right). 
\end{align*}
By virtue of Lemmas \ref{l-inter-p} and \ref{bd-m-vk-r}, 
there exists a constant $C$ being independent of $r$ such that 
$\av{F(\vk_r)} \leq C(T, \gm_0)$ for any $(x,t) \in [-r,r] \times [0, T]$. 
Hence, for any $(x,t) \in [-r,r] \times [0,T]$, we have 
\begin{align*}
e^{-C T} \leq \av{\pd_x \gm_r(x,t)} \leq e^{C T}. 
\end{align*}
Next we turn to the proof of \eqref{ele-bd-2}. 
Here we have 
\begin{align} \label{bd-1}
\pd^m_x F(\vk_r) - \av{\pd_x \gm_r}^m \pd^m_s F(\vk_r) 
 = P(\av{\pd_x \gm_r}, \cdots, \pd^{m-1}_x \av{\pd_x \gm_r}, F(\vk_r), \cdots, \pd^{m-1}_s F(\vk_r)). 
\end{align}
Suppose that there exist constants $C_j(T, \gm_0)$ being independent of $r$ such that 
\begin{align*}
\sup_{(x,t) \in [0,r] \times [0,T]} \av{\pd^j_x \av{\pd_x \gm_r} } \leq C_j (T,\gm_0) 
\end{align*}
holds for any $0 \leq j \leq m-1$. Then \eqref{bd-1} implies 
\begin{align*}
\av{\pd^m_x F(\vk_r)} < C, 
\end{align*}
where the constant $C$ is independent of $r$. 
Differentiating the equation \eqref{eq-le-1} with respect to $x$, we have 
\begin{align*}
\pd_t \pd^m_x \av{\pd_x \gm_r} 
  = F(\vk_r) \pd^m_x \av{\pd_x \gm_r} 
     + \sum^{m}_{j=1} \mbox{}_m C_j \pd^j_x F(\vk_r) \pd^{m-j}_x \av{\pd_x \gm_r}. 
\end{align*}
Thus $\pd^m_x \av{\pd_x \gm_r}$ is a solution of 
\begin{align} \label{od-eq-2}
\begin{cases}
& \pd_t v = F(\vk_r) v + G, \\
& v(0) = 0. 
\end{cases}
\end{align}
Then we see that there exists a constant $C_2(T,\gm_0)$ being independent of $r$ such that $\av{v} \leq C_2$. 
This gives us the conclusion of Lemma \ref{ele-bd-lm}. 
\end{proof}

In order to state our main result precisely, we define the following: 
\begin{dfn}
Let $\gm(x) : \R \to \R^2$ be a planar curve. 
$\gm$ is called proper if 
$\displaystyle{\lmt{\av{x} \to + \infty}{\av{\gm(x)}}}= + \infty$. 
\end{dfn}
We are now in a position to prove an existence of a classical solution to 
\begin{align*} 
\begin{cases}
& \pd_t \gm = \left( -2 \pd^2_s \vk - \vk^3 + \lm^2 \vk \right) \bn, \\
& \gm(x,0)= \gm_0(x) 
\end{cases}
\end{align*}
for any finite time: 
\begin{thm} \label{main-thm-1}
Let $\gm_0(x)$ be a proper planar curve satisfying \eqref{cond-1}--\eqref{cond-4}. 
Then there exist a family of smooth proper planar curves 
$\gm(x,t) : \R \times [0, \infty) \to \R^2$ satisfying \eqref{s-s}. 
Moreover the following holds{\rm :} 
\begin{enumerate}
\item[{\rm (i)}] There exists a positive constant $K$ being independent of $t$ such that 
\begin{align} \label{ub}
\max{\left\{ \Lns{\pd^n_s \vk(t)}{2}, \quad 
\Lns{\pd^n_s \vk(t)}{\infty} \right\}} < K 
\end{align}
for any $n \in \N \cup \{0\}$, where $\vk$ denotes the curvature of $\gm$.  
\item[{\rm (ii)}] Let ${\bf e}= (0,1)$. As $\av{x} \to \infty$, 
\begin{align} \label{approach}
\gm(x,t) \cdot {\bf e} \to 0, \quad \pd_x \gm(x,t) \cdot {\bf e} \to 0 
\end{align}
for any $t>0$. 
\end{enumerate} 
\end{thm}
\begin{proof}
To begin with, we prove a long time existence of a classical solution of \eqref{s-s} 
by making use of Arzel\`a-Ascoli's theorem. 
Let us fix $N>M$ and $T>0$ arbitrarily. 
First we show that $\{ \gm_r \}_{r>N}$ is uniformly bounded on $[-N,N] \times [0,T]$ 
with respect to $r$. 
Let $r-1>N$. For any $(x,t) \in [-N,N] \times [0, T]$, it holds that 
\begin{align*}
\av{\gm_r(x,t)} & \leq \av{\gm_r(x,0)} + \int^T_0 \av{\pd_t \gm_r(x,\tau)} \, d\tau \\
 & \leq \av{\gm_0(x)} 
     + \int^T_0 \left\{ 2 \Lns{\pd^2_s \vk_r(\tau)}{\infty} + \Lns{\vk_r(\tau)}{\infty}^3 
                              + \lm^2 \Lns{\vk_r(\tau)}{\infty} \right\} \, d \tau 
 < C(\gm_0, N, T, \lm)
\end{align*}
Since 
\begin{align*}
\pd^m_x \gm_r - \av{\pd_x \gm_r}^m \pd^m_s \gm_r 
 = P(\av{\pd_x \gm_r}, \cdots, \pd^{m-1}_x \av{\pd_x \gm_r}, \gm_r, \cdots, \pd^{m-1}_s \gm_r), 
\end{align*}
Lemma \ref{ele-bd-lm} yields that there exists a positive constant $C(N,T,\gm_0)$ such that 
\begin{align*}
\av{\pd^m_x \gm_r(x,t)} \leq C(N,T,\gm_0). 
\end{align*}
Moreover, since $\Lns{\pd^m_s \vk_r}{2} < \infty$ for any $m \in \N$, 
we have 
\begin{align*}
\av{\pd_t \gm_r(x,t)} \leq C(N,T,\gm_0). 
\end{align*}

Next we prove an equi-continuity of $\{ \gm_r \}_{r >N}$ with respect to $r$. 
From the uniform boundedness of $\{ \gm_r \}_{r>N}$, we have 
\begin{align*}
\av{\gm_r(x,t) - \gm_r(y,\tau)} 
 & \leq \av{\gm_r(x,t) - \gm_r(y,t)} + \av{\gm_r(y,t) - \gm(y,\tau)} \\
 & \leq \int^x_y \av{\pd_x \gm_r(\xi, t)} \, d \xi 
         + \int^t_\tau \av{\pd_t \gm_r (y, \rho)} \, d \rho \\
 & \leq \int^x_y \av{\pd_x \gm_r(\xi, t)} \, d \xi 
         + \int^t_\tau \av{F^\lm(y,\rho)} \, d \rho 
   \leq C_1 \av{x-y} + C_2 \av{t - \tau}, 
\end{align*}
where the constants $C_1$ and $C_2$ are independent of $r$. 
Similarly we see that  
\begin{align*}
\av{\pd_t \gm_r(x,t) -\pd_t \gm_r(y,\tau)} 
 & \leq C_3 \av{x-y} + C_4 \av{t - \tau}, \\
\av{\pd^m_x \gm_r(x,t) -\pd^m_x \gm_r(y,\tau)} 
 & \leq C_5 \av{x-y} + C_6 \av{t - \tau}, 
\end{align*}
where $m$ is any natural number. 
Thus the sequence $\{ \gm_r \}_{r>N}$ is equi-continuous. 
Therefore, Arzel\`a-Ascoli's theorem and a diagonal method imply that 
there exist a subsequence $\{ \gm_{r_j} \}_{j=1}^\infty$ and 
a family of smooth planar curves $\gm$ defined on $[-N,N] \times [0,T]$ such that 
\begin{align*}
& \su{(x,t) \in [-N,N] \times [0,T] }{\av{\pd^m_x \gm_{r_j}(x,t) - \pd^m_x \gm(x,t)}} \to 0, \\
& \su{(x,t) \in [-N,N] \times [0,T]}{\av{\pd_t \gm_{r_j}(x,t) - \pd_t \gm(x,t)}} \to 0, 
\end{align*}
as $j \to \infty$. Since $\gm_{r_j}$ satisfies $({\rm SS}_{r_j})$ for any $j$, 
we see that $\gm$ satisfies \eqref{s-s} on $[-N,N] \times [0,T]$. 

We can verify that $\gm$ is defined on $\R \times [0,\infty)$. 
Indeed, let $\{ R_j \}^\infty_{j=1}$ be a sequence with $R_j > M$ and $R_j \to \infty$ as $j \to \infty$. 
Set $Q_j = (-R_j, R_j) \times [0, R_j)$. 
Then there exist a subsequence $\{ \gm_{r_{1j}} \}^\infty_{j=1} \subset \{ \gm_r \}_{r > R_1}$ and 
a planar curve $\gm$ defined on $Q_1$ such that $\gm_{r_{1j}} \to \gm$ as $j \to \infty$. 
Moreover $\gm$ satisfies \eqref{s-s} on $Q_1$. 
Next, for $\{ \gm_{r_{1j}} \}_{r_{1j}>R_2}$, there exists a subsequence 
$\{ \gm_{r_{2j}} \}^\infty_{j=1} \subset \{ \gm_{r_{1j}} \}_{r_{1j} > R_2}$ 
such that $\gm_{r_{2j}} \to \gm$ in $Q_2$ as $j \to \infty$. 
Similarly we observe that, for any $m \in \N$, there exists a subsequence 
$\{ \gm_{r_{mj}} \}^\infty_{j=1} \subset \{ \gm_{r_{m-1 j}} \}_{r_{m-1 j} > R_m}$ 
such that $\gm_{r_{mj}} \to \gm$ in $Q_m$ as $j \to \infty$. 
Letting $\{ \gm_{r_{nn}} \}^\infty_{n=1}$, we see that 
$\gm$ is defined on $\R \times [0,\infty)$ 
and satisfies \eqref{s-s} on $(-R,R) \times [0,R)$ for any $R > M$. 

Next we shall prove that $\gm(x,t)$ is a smooth proper curve for any $t>0$. 
Let $R>M$ fix arbitrarily and define a strip domain as follows:  
\begin{align}
S(R):= \{ (x_1, x_2) \mid -R \leq x_1 \leq R \}. 
\end{align}
Then there exists $r>R$ such that 
\begin{align}
-\phi_0(-r) < -R < R < \phi_0(r). 
\end{align}
For such $r>R$, we observe that 
\begin{align} \label{lower-est}
\mathcal{H}^1(\gm_r(t) \cap S(R)) \geq 2R. 
\end{align}
For the curve $\gm_r$ is fixed at the both $(\phi_0(-r),0)$ and $(\phi_0(r), 0)$. 
Since $\gm_{r}(x,t)$ converges to $\gm(x,t)$ smoothly along a sequence $\{r_j \}_j$, 
the inequality \eqref{lower-est} implies that 
\begin{align} \label{lower-est-2}
\mathcal{H}^1(\gm(t) \cap S(R)) \geq 2 R 
\end{align}
for any $R>0$. 

Next we turn to the estimate \eqref{ub}. 
By virtue of Lemma \ref{bd-vk-r}, we see that there exists a constant $C_0 > 0$ being independent
of $r$ and $t$ such that 
\begin{align} \label{ineq-1}
\Lns{\vk_r(t)}{2} < C_0. 
\end{align}
The inequality is equivalent to 
\begin{align} \label{ineq-2}
\Lns{\pd^2_s \gm_r(t)}{2} < C_0. 
\end{align}
Combining Lemmas \ref{l-est-deri-4}, \ref{l-est-deri-n}, and \ref{l-conv-0} with the inequality \eqref{ineq-2}, 
we observe that there exists a constant $C_n>0$ being independent of $r$ and $t$ such that 
\begin{align} \label{ineq-3}
\Lns{\pd^{n+2}_s \gm_r(t)}{2} < C_n, 
\end{align}
where $n$ is any non-negative integer. 
The inequality \eqref{ineq-3} yields that 
\begin{align} \label{ineq-4}
\Lns{\pd^n_s \vk_r(t)}{2} < C_n
\end{align}
holds for each non-negative integer $n$. 
By using Lemma \ref{l-inter-p}, we obtain 
\begin{align} \label{ineq-5}
\max{}{\left\{\Lns{\pd^{n+2}_s \gm_r(t)}{\infty}, \Lns{\pd^{n+2}_s \vk_r(t)}{\infty} \right\}} 
 < \tilde{C}_n
\end{align}
for each $n$, where $\tilde{C}_n$ is independent of $r$ and $t$. 
Therefore we obtain \eqref{ub}. 

Finally we prove \eqref{approach}. 
Let $T>0$ fix arbitrarily. 
First we prove that $\gm(x,t) \cdot {\bf e}$ converges to $0$ as $\av{x} \to \infty$ for 
any $0 < t \leq T$, where ${\bf e}= (0,1)$. Then, by virtue of Lemma \ref{ele-bd-lm}, we have 
\begin{align*}
\dfrac{d}{dt} \Ln{\gm(t) \cdot {\bf e}}{\R}^2 
 &= 2 \int^\infty_{-\infty} (\pd_t \gm(x,t) \cdot {\bf e} ) (\gm(x,t) \cdot {\bf e} ) \, dx \\
 &= 2 \int^\infty_{-\infty} (\pd_t \gm(x,t) \cdot {\bf e} ) \av{\pd_x \gm(x,t)}^\frac{1}{2} \cdot
                                   (\gm(x,t) \cdot {\bf e} ) \av{\pd_x \gm(x,t)}^{-\frac{1}{2}} \, dx \\
 &\leq 2 \left\{ \int_\gm \av{\pd_t \gm(x,t) \cdot {\bf e}}^2 \, ds \right\}^\frac{1}{2} 
            \left\{ \int^\infty_{-\infty} \av{\psi(x,t)}^2 \av{\pd_x \gm(x,t)}^{-1} \, dx \right\}^\frac{1}{2} \\
 &\leq C \Lns{\pd_t \gm(t)}{2} \Ln{\gm(t) \cdot {\bf e}}{\R}. 
\end{align*}
Using Lemmas \ref{bd-vk-r} and \ref{bd-m-vk-r}, we obtain the following: 
\begin{align} \label{2nd-1}
\dfrac{d}{dt} \Ln{\gm(t) \cdot {\bf e}}{\R}^2 
 \leq \Ln{\gm(t) \cdot {\bf e}}{\R}^2 + C, 
\end{align}
where $C$ depends only on $\gm_0$ and $T$. 
The inequality \eqref{2nd-1} implies that $\gm(x,t) \cdot {\bf e}$ satisfies 
\begin{align}
\Ln{\gm(t) \cdot {\bf e}}{\R}^2 \leq \left( \Ln{\gm_0 \cdot {\bf e}}{\R}^2 + C \right) e^T
\end{align}
for any $0 < t \leq T$. Therefore we see that $\gm(x,t) \cdot {\bf e} \to 0$ as $\av{x} \to \infty$ 
for any $0 < t \leq T$. 
Next we prove a convergence of $\pd_x \gm(x,t) \cdot {\bf e}$ as $\av{x} \to \infty$. 
Making use of Lemma \ref{ele-bd-lm}, we have the following: 
\begin{align*}
\dfrac{d}{dt} \Ln{\pd_x \gm(t) \cdot {\bf e}}{\R}^2 
 &= 2 \int^\infty_{-\infty} (\pd_x \pd_t \gm(x,t) \cdot {\bf e} ) ( \pd_x \gm(x,t) \cdot {\bf e} ) \, dx \\
 &= 2 \int^\infty_{-\infty} (\pd_s \pd_t \gm(x,t) \cdot {\bf e} ) \av{\pd_x \gm(x,t)}^\frac{1}{2} \cdot
                                   (\pd_x \psi(x,t) ) \av{\pd_x \gm(x,t)}^{\frac{1}{2}} \, dx \\
 &\leq 2 \left\{ \int_\gm \av{\pd_s \pd_t \gm(x,t) \cdot {\bf e}}^2 \, ds \right\}^\frac{1}{2} 
            \left\{ \int^\infty_{-\infty} \av{\pd_x \psi(x,t)}^2 \av{\pd_x \gm(x,t)} \, dx \right\}^\frac{1}{2} \\
 &\leq C \Lns{\pd_s \pd_t \gm(t)}{2} \Ln{\pd_x \gm(t) \cdot {\bf e}}{\R}. 
\end{align*}
Along the same line as above, we see that $\pd_x \gm(x,t) \cdot {\bf e} \to 0$ as $\av{x} \to \infty$ 
for any $0 < t \leq T$. 
\end{proof}
In the rest of this paper, we prove that the solution of \eqref{s-s} obtained by 
Theorem \ref{main-thm-1} converges to a stationary solution as $t \to \infty$. 
Moreover we see that the stationary solution is a line or a borderline elastica 
(see Figure \ref{figu-1}): 
\begin{figure}
\begin{center}
\includegraphics[height=4.0in, width=5.0in]{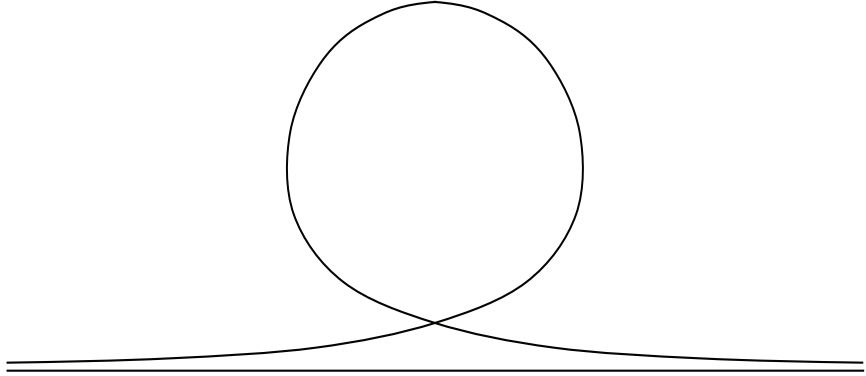}
\caption{An example of a curve whose curvature is given by \eqref{s-curve-1}.}  \label{figu-1}
\end{center}
\end{figure}%
\begin{thm} \label{main-thm-2}
Let $\gm(x,t) : \R \times [0, \infty) \to \R^2$ be a solution of \eqref{s-s} obtained by 
Theorem \ref{main-thm-1}. Then there exist sequences $\{ t_j \}^\infty_{j=1}$ 
and $\{ p_j \}^\infty_{j=1}$ and a smooth proper curve $\hat{\gm} : \R \to \R^2$ such that 
$\gm(\cdot,t_j) - p_j$ converges to $\hat{\gm}(\cdot)$ as $t_j \to \infty$ 
up to a reparametrization. The curvature $\hat{\vk}$ satisfies 
\begin{align} \label{EL-eq}
2 \pd^2_{s} \hat{\vk} + \hat{\vk}^3 - \lm^2 \hat{\vk}=0
\end{align}
and 
\begin{align} \label{el-bd}
\int_{\hat{\gm}} \hat{\vk}^2 \, ds < \infty. 
\end{align}
Moreover $\hat{\vk}$ is given by either 
\begin{align}
\hat{\vk} \equiv 0
\end{align}
or 
\begin{align} \label{s-curve-1}
\hat{\vk}(s)= 
\begin{cases}
k(s-s_0) & \text{for} \quad s>s_0, \\
k(-s+s_0) & \text{for} \quad s<s_0 
\end{cases}
\end{align}
for some $s_0 \in \R$, where $k(s)$ is the solution of either 
\begin{align} \label{form-1}
\begin{cases}
\dfrac{d k}{d s}= - \sqrt{- \dfrac{k^4}{4} + \dfrac{\lm^2}{2} k^2} 
\quad \text{for} \quad s \in \R, \\
k(0)= \sqrt{2} \av{\lm} 
\end{cases}
\end{align}
or 
\begin{align} \label{form-2}
\begin{cases}
\dfrac{d k}{d s}= \sqrt{- \dfrac{k^4}{4} + \dfrac{\lm^2}{2} k^2} 
\quad \text{for} \quad s \in \R, \\
k(0)= -\sqrt{2} \av{\lm}. 
\end{cases}
\end{align}
\end{thm} 
\begin{proof}
From \eqref{ub}, it follows that $\pd^n_s \vk(\cdot,t)$ is uniformly continuous with respect to $t$. 
Furthermore, the fact \eqref{ub} implies that, as $\av{x} \to \infty$, 
\begin{align} \label{conv-1}
\av{\pd^{n+2}_s \gm(x,t)} \to 0, \quad \av{\pd^n_s \vk(x,t)} \to 0 
\end{align}
for any $t>0$. 
Then, along the same line as in Section \ref{convergence}, we are able to prove that 
\begin{align} \label{conv-2}
\Lns{\pd^n_s \pd_t \gm(t)}{2} \to 0 \quad \text{as} \quad t \to \infty. 
\end{align}
Here we reparametrize $\gm$ by its arc length, i.e., $\gm= \gm(s,t)$. 
Then, \eqref{lower-est-2} implies that $\gm(s,t)$ is defined on $[0,L] \times [0,\infty)$ for 
any $L \geq 2R$. 
In the following, let $L \geq 2R$ fix arbitrarily. 
For the curve $\gm=\gm(s,t) : [0,L] \times [0,\infty) \to \R^2$, first we observe that 
\begin{align}
\av{\gm(s,t) - \gm(0,t)} \leq s \leq L 
\end{align}
for any $(s,t) \in [0,L] \times [0, \infty)$. 
Thus we see that $\gm(s,t) - \gm(0,t)$ is uniformly bounded with respect to $t$. 
It is easy to check that $\gm(s,t) - \gm(0,t)$ is equi-continuous 
with respect to $t$. 
Indeed, since it holds that 
\begin{align*}
\av{\left\{ \gm(s_1,t)-\gm(0,t) \right\} - \left\{ \gm(s_2,t)-\gm(0,t) \right\} }
\leq \av{s_1 - s_2}, 
\end{align*}
if $\av{s_1-s_2}<\vd$, then we have  
\begin{align*}
\av{\left\{ \gm(s_1,t)-\gm(0,t) \right\} - \left\{ \gm(s_2,t)-\gm(0,t) \right\} } 
 < \ve 
\end{align*}
for any $t>0$. 
Moreover, with the aid of \eqref{ub}, we verify that 
\begin{align*}
\av{\pd^n_s \vk(s_1,t) - \pd^n_s \vk(s_2,t)} 
 \leq \int^{s_1}_{s_2} \av{\pd^{n+1}_s \vk(\vs,t)} \, d\vs < C \av{s_1- s_2}. 
\end{align*}
Therefore, by virtue of Arzel\`a-Ascoli's theorem and a diagonal method, we see that 
there exist a sequence $\{t_j \}^\infty_{j=1}$, a planar curve $\hat{\gm}(\cdot)$, 
and a function $\hat{\vk}(\cdot)$ such that 
\begin{gather}
\gm(\cdot,t_j) - \gm(0,t_j) \to \hat{\gm}(\cdot), \\
\pd^n_s \vk(\cdot, t_j) \to \pd^n_{s} \hat{\vk}(\cdot) \label{vk-conv}
\end{gather}
for all $n \in \N \cup \{0\}$ as $t_j \to \infty$. 
This implies that the curve $\hat{\gm}(\cdot)$ is smooth and 
there exists a sequence $\{p_j \}^\infty_{j=1} \subset \R^2$, 
with $p_j= \gm(0, t_j)$, such that 
\begin{align*}
\gm(\cdot, t_j) - p_j \to \hat{\gm}(\cdot) 
\end{align*}
as $t_j \to \infty$. Furthermore, it follows from \eqref{conv-2} and \eqref{vk-conv} 
that $\hat{\vk}$ satisfies \eqref{EL-eq}. 
Since $\gm(s,t)$ converges to $\hat{\gm}$ along a sequence $\{ t_j \}^\infty_{j=1}$ on any 
compact set $[0,L]$, the estimate \eqref{lower-est-2} yields that the limiting curve 
$\hat{\gm}$ is also a smooth proper curve. 
Moreover \eqref{el-bd} follows from \eqref{ub} letting $t \to \infty$ 
along $\{t_j \}^\infty_{j = 1}$. 

Finally we derive a representation formula of $\hat{\vk}$. 
From \eqref{EL-eq}, we obtain 
\begin{align} \label{dk-ds}
\left( \dfrac{d \hat{\vk}}{ds} \right)^2 = -\dfrac{\hat{\vk}^4}{4} + \dfrac{\lm^2}{2}\hat{\vk}^2 +C, 
\end{align}
where $C$ is an arbitrary constant. 
A standard theory of ordinary differential equations yields that 
the fact \eqref{el-bd} implies $C=0$. 
Then it is clear that $\hat{\vk} \equiv 0$ satisfies \eqref{dk-ds}. 
If $\hat{\vk}$ is non-trivial, then there exists a point $s=s_0$ such that $d \hat{\vk}/ds$ vanishes. 
Therefore we obtain the conclusion. 
\end{proof}

\begin{rem} \label{rem-thm-2}
Along the same line as in the proof of Theorem \ref{main-thm-2}, 
we can also prove that, for any sequences $\{ t_j \}_j$ and $\{ p_j \}_j \subset \R^2$ 
with $p_j \in \gm(\cdot, t_j)$, 
there exist a subsequence $\{ t_{j_k} \} \subset \{ t_j \}$ and a stationary solution $\hat{\gm}$ 
such that $\gm(\cdot, t_{j_k}) - p_{j_k} \to \hat{\gm}(\cdot)$ as $t_{j_k} \to \infty$. 
Indeed, the claim is proved by applying our argument to $\gm(\cdot, t_j)$. 
\end{rem}
We define an index of $\gm$ as follows: 
\begin{align*}
i(\gm)= \int_\gm \vk \, ds. 
\end{align*}
Regarding the index $i(\gm)$, we prove that $i(\gm)$ is invariant under 
the shortening-straightening flow for any finite time. 
\begin{lem} \label{wind-invari}
Let $\gm(x,t)$ be a solution of \eqref{s-s}. 
Then $i(\gm)$ is invariant for any finite time $t>0$. 
\end{lem}
\begin{proof}
By virtue of Lemma \ref{k-flow}, we observe that 
\begin{align*}
\dfrac{d}{dt} i(\gm)  
 &= \int_\gm \vk_t \, ds + \int_\gm \vk \, \pd_t ds \\
 &= - \int_\gm \pd^2_s F^\lm \, ds 
  = - \left[ \pd_s F^\lm \right]^\infty_{-\infty}
  = - \left[ 2 \pd^3_s \vk + 3 \vk^2 \pd_s \vk - \lm^2 \pd_s \vk \right]^\infty_{-\infty}.  
\end{align*}
Since Lemma \ref{bd-m-vk-r} yields that 
\begin{align*}
\int_\gm \left( \pd^m_s \vk \right)^2 \, ds < \infty
\end{align*}
for any $m \in \N$ and finite time $t>0$, we see that, as $\av{x} \to \infty$, 
\begin{align*}
\pd_s F^\lm \to 0.  
\end{align*}
\end{proof}

With the aid of Lemma \ref{wind-invari} and Remark \ref{rem-thm-2}, 
we can characterize a dynamical aspect of $\gm$ starting from $\gm_0$ with $i(\gm_0) \neq 0$. 

\begin{thm}\label{main-thm-3}
Let $\gm_0 : \R \to \R^2$ be a smooth planar curve satisfying \eqref{cond-1}--\eqref{cond-4}. 
Let $\gm : \R \times [0, \infty) \to \R^2$ be a solution of \eqref{s-s} obtained by 
Theorem \ref{main-thm-1}. If $i(\gm_0) \neq 0$, then there exists at least one sequence 
$\{t_j \}_j$ with $t_j \to \infty$ such that $\gm(\cdot, t_j)$ converges to a borderline elastica 
as $t_j \to \infty$. 
\end{thm}
\begin{proof}
Let $\{t_j \}_j$ be an arbitral sequence with $t_j \to \infty$. 
If $i(\gm_0) \neq 0$, then Lemma \ref{wind-invari} implies that $\gm(\cdot,t)$ always contains 
at least one loop part $l(\gm(t))$. 
Let us define a sequence $\{ p_j \}_j \subset \R^2$ as  
\begin{align} \label{pj-def}
p_j \in l(\gm(t_j))
\end{align}
for each $j \in \N$. Then, as we stated in Remark \ref{rem-thm-2}, 
there exist a subsequence $\{ t_{j_k} \} \subset \{ t_j \}$ and a stationary solution $\hat{\gm}$ 
such that $\gm(\cdot, t_{j_k}) - p_{j_k} \to \hat{\gm}(\cdot)$ as $t_{j_k} \to \infty$. 
By virtue of \eqref{pj-def}, the curve $\hat{\gm}$ can not be a straight line. 
Therefore Theorem \ref{main-thm-2} gives us the conclusion. 
\end{proof}

\vspace{0.4cm}

\noindent
{\bf Acknowledgements}

\vspace{0.1cm}

The first author was partially supported by the Fondazione CaRiPaRo Project 
{\it Nonlinear Partial Differential Equations: models, analysis, and 
control-theoretic problems}. 
The second author was partially supported by Grant-in-Aid for Young Scientists (B) (No. 24740097). 


\vspace{0.6cm}




\begin{thebibliography}{00}

 \bibitem[1]{ange} S. B. Angenent, 
{\it On the formation of singularities in the curve shortening flow}, 
J. Differential Geom. {\bf 33} (1991), no. 3, 601--633.   

 \bibitem[2]{b-m-n}G. Bellettini, C. Mantegazza, and M. Novaga, 
{\it Singular perturbations of mean curvature flow}, J. Differential Geom. 
{\bf 75} (2007), 403--431. 

 \bibitem[3]{dziuk} G. Dziuk, E. Kuwert, and R. Sch\"atzle, 
{\it Evolution of elastic curves in $\R^n$: existence and computation}, 
SIAM J. Math. Anal. {\bf 33} (2002), 1228--1245. 

 \bibitem[4]{gage} M. E. Gage, {\it Curve shortening makes convex curves circular}, 
Invent. Math. {\bf 76} (1984), no. 2, 357--364. 

 \bibitem[5]{gray} M. A. Grayson, {\it The shape of a figure-eight under the curve shortening flow}, 
 Invent. Math. {\bf 96} (1989), no. 1, 177--180. 

 \bibitem[6]{koiso} N. Koiso, 
{\it On the motion of a curve towards elastica}, 
Acta de la Table Ronde de G\'eom\'etrie Diff\'erentielle, (1996), 403--436, 
S\'emin. Congr. 1, Soc. Math. France, Paris. 
 
 \bibitem[7]{linner} A. Linn\'er, {\it Some properties of the curve straightening flow in the plane}, 
Trans. Amer. Math. Soc. {\bf 314} (1989), no. 2, 605--618. 

  \bibitem[8]{lunardi} A. Lunardi, {\it Analytic Semigroup and Optimal Regularity in Parabolic Problems}, 
Progress in Nonlinear Differential Equations and Their Applications {\bf 16} (1995), Birkh\"auser. 
 
 \bibitem[9]{okabe-1} S. Okabe, {\it The motion of elastic planar closed curves under the area-preserving 
condition}, Indiana Univ. Math. J. {\bf 56} (2007), no. 4, 1871--1912. 

 \bibitem[10]{polden} A. Polden, 
{\it Curves and Surfaces of Least Total Curvature And Fourth-Order Flows}, 
Dissertation University of Tuebingen (1996). 

 \bibitem[11]{wen-1} Y. Wen, {\it $L^2$ flow of curve straightening in the plane}, 
 Duke Math. J. {\bf 70} (1993), no. 3, 683--698. 
 
 \bibitem[12]{wen-2} Y. Wen, {\it Curve straightening flow deforms closed plane curves with nonzero 
rotation number to circles}, J. Differential Equations {\bf 120} (1995), no. 1, 89--107. 

\end{thebibliography}
\end{document}